\documentclass[aip,superscriptaddress]{revtex4}
\usepackage[utf8]{inputenc} 
\usepackage[T1]{fontenc}
\usepackage{amsmath,amssymb,amsfonts}
\usepackage{bbold}
\usepackage{color}
\usepackage{url}
\usepackage{graphicx}
\usepackage{tikz}
\usepackage{epstopdf}

\usepackage{hyperref}

\usepackage{amsthm}

\theoremstyle{plain}
\newtheorem{theorem}{Theorem}[section]

\newtheorem{corollary}[theorem]{Corollary}
\newtheorem{proposition}[theorem]{Proposition}

\theoremstyle{definition}
\newtheorem{definition}[theorem]{Definition}
\newtheorem{example}[theorem]{Example}
\newtheorem{conjecture}[theorem]{Conjecture}
\newtheorem{remark}[theorem]{Remark}

\newcommand{\diag}{{\rm diag\,}}
\newcommand{\sign}{{\rm sign\,}}

\newcommand{\RE}{{\rm Re\,}}

\newcommand{\eins}{\leavevmode\hbox{\small1\kern-3.8pt\normalsize1}}

\begin{document}

\title[Products of Complex and Hermitian Matrices]{Products of Complex Rectangular and Hermitian Random Matrices}
\author{Mario Kieburg}\email[]{m.kieburg@unimelb.edu.au}
\affiliation{School of Mathematics and Statistics, The University of Melbourne, 813 Swanston Street, Parkville, Melbourne VIC 3010, Australia}

\newcommand{\corr}[1]{{\color{red}#1}}
\newcommand{\tim}[1]{{\color{blue}#1}}


\begin{abstract}
Products and sums of random matrices have seen a rapid development in the past decade due to various analytical techniques available. Two of these are the harmonic analysis approach and the concept of polynomial ensembles. Very recently, it has been shown for products of real matrices with anti-symmetric matrices of even dimension that the traditional harmonic analysis on matrix groups developed by Harish-Chandra et al. needs to be modified when considering the group action on general symmetric spaces of matrices. In the present work, we consider the product of complex random matrices with Hermitian matrices, in particular the former can be also rectangular while the latter has not to be positive definite and is considered as a fixed matrix as well as a random matrix. This generalises an approach for products involving the Gaussian unitary ensemble (GUE) and circumvents the use there of non-compact group integrals. We derive the joint probability density function of the real eigenvalues and, additionally, prove transformation formulas for the bi-orthogonal functions and kernels.
\ \\
{\bf Keywords:} products of independent random matrices; polynomial ensemble; 
multiplicative convolution; P\'olya frequency functions;
spherical transform; bi-orthogonal ensembles.\\
\ \\
{\bf MSC:} 15B52, 42C05
\end{abstract}

\maketitle

\section{Introduction}\label{sec:intro}

There are two fast developments in the past years regarding products and sums of random matrices. One concerns the macroscopic level densities that can be elegantly computed via the free probability  approach~\cite{Speicher2011} introduced by Voiculescu~\cite{Voiculescu1991}. The other development has the local spectral statistics in its focus for which the concept of polynomial ensembles~\cite{KZ2014} has proven fruitful, see also Definition~\ref{def:ensembles}. The latter enjoys the analytical tools of bi-orthogonal functions and determinantal point processes~\cite{Borodin1998}. Both aforementioned developments rely on the unitary (or orthogonal) bi-invariance of at least one of the two multiplied random matrices, meaning the ensemble is invariant under left- as well as right-multiplication of unitary (orthogonal) matrices. The considerations made in the present work exploits the same requirement.

Formerly, only products of complex Ginibre matrices~\cite{AB2012, AKW2013,KZ2014}, then of complex rectangular Gaussian matrices~\cite{AIK2013,IK2014,KS2014,ARRS2016,LW2016} and inverse Ginibre matrices~\cite{Forrester2014,ARRS2016}, and later of truncated unitary matrices (complex Jacobi ensemble)~\cite{ABKN2014,ARRS2016,KKS2016,LW2016} have been considered, see~\cite{AI2015} for a review on this development. These calculations have been possible due to known group integrals that are involved. Only a few years later the concept of harmonic analysis~\cite{Helgason2000} has been combined with these new developments to enlarge the class of non-classical ensembles to those of P\'olya ensembles~\cite{FKK2017,Kieburg2017,KFI2019} (originally known as polynomial ensembles of derivative type~\cite{KK2016,KK2019,KR2016}), see Definition~\ref{def:ensembles}, which also comprise some Muttalib-Borodin ensembles~\cite{Borodin1998,Muttalib1995}, for instance.

The name P\'olya ensemble has been dubbed to these ensembles due to their intimate relation to P\'olya frequency functions, that originated in approximation theory~\cite{Polya1913,Polya1915}. Because of the same importance of P\'olya frequency functions for sums of random matrices~\cite{KR2016,FKK2017}, we speak of multiplicative and additive P\'olya ensembles. These two kinds indeed differ by a subtle, but important feature which will not be specified here; we refer interested readers to~\cite{FKK2017}. Let us underline, throughout the present article we consider multiplicative P\'olya ensembles, only. Therefore, we omit the prefix ``multiplicative''.

P\'olya ensembles have two crucial advantages. Firstly, the multiplicative convolution is closed on their class so that they form a semi-group with the multiplication of a Haar distributed unitary matrix (also known as circular unitary matrices (CUE)) as the unit element. Secondly, the whole statistics essentially depends on a single univariate weight function. This fact simplifies the analysis drastically. For example, the harmonic analysis on matrix spaces reduces to its univariate counterpart. For products of complex random matrices this means that the spherical transform~\cite{Helgason2000,KK2016,KK2019} condenses to the Mellin transform in the one-dimensional case.

Very recently, products of real square matrices with real antisymmetric matrices and of quaternion matrices with anti-selfdual anti-Hermitian matrices have been studied~\cite{FILZ2019,KFI2019}. These works have followed an article on products of complex Ginibre matrices with a Gaussian unitary ensemble (GUE)~\cite{FIL2018}. Therein it has been shown that the natural group action of the matrix group $G$ on the corresponding Lie-algebra of its maximal compact subgroup might be indeed analytically feasible. To extend these results to more general ensembles, harmonic analysis seems to be again the ideal tool. Yet, for its successful application to the product of real asymmetric with real anti-symmetric matrices of even dimension a modification of the original theory by Harish-Chandra~\cite{Helgason2000} has been crucial~\cite{FILZ2019}. This can be all expected to hold true for the product of complex matrices with Hermitian matrices, as it is shown in the present work.

After introducing the necessary notation in Section~\ref{sec:pre}, we extend the concept of harmonic analysis to the set of Hermitian matrices with a fixed rank in Section~\ref{sec:spherical}. The spherical function and transform, introduced in this section, is here not anymore the one of the standard Fourier analysis, that corresponds to the additive convolution on Hermitian matrices. It reflects the multiplicative convolution and looks very similar to the original spherical function with an important difference, we need two copies of the complex plane for each of the ``Mellin-Fourier'' (frequency) parameter that encodes the sign of the eigenvalue. Furthermore we extend the harmonic analysis on the complex general linear group~\cite{Helgason2000,KK2016} to complex rectangular matrices. Quite naturally, orthogonal projections and inclusions play a particular outstanding role as has been already highlighted in~\cite{IK2014}.

Once the theoretical framework has been introduced, we consider the actual goal of the present article namely the product of complex rectangular and Hermitian random matrices. In Section~\ref{sec:mult.G}, we consider first the product of P\'olya ensembles of complex rectangular matrices with fixed Hermitian matrices. Therein, we compute the joint probability density, a corresponding set of bi-orthogonal functions, and an expression for the kernel. As a second example, we study the product of a P\'olya ensemble on the rectangular matrices with a polynomial ensemble drawn from the Hermitian matrices, in Section~\ref{sec:mult.H}. The change of the joint probability density, the bi-orthogonal functions, and the kernel will be the primary aim. Such transformation formulas have been considered in other works, too, that involve sums of random matrices~\cite{Kuijlaars2016,CKW2015,Kieburg2017} as well as products~\cite{KKS2016,KK2019,Kuijlaars2016,CKW2015}. Let us underline that all calculations are done at finite matrix dimensions, and that it is not the goal of the present work to investigate any limiting statistics. A technical work on a general product involving a GUE is currently in preparation~\cite{Kieburg2019}.

In Section~\ref{sec:conclusio}, we summarize and give a brief outlook on open problems. The details of the more extensive proofs are given in the Appendices.

\section{Preliminaries}\label{sec:pre}

In the ensuing sections, we make use of the particular matrix sets and $L^1$-function which we define here. For this purpose, we first of all introduce  the orthogonal projection/inclusion matrices of size $n\times m$,
\begin{equation}\label{proj-incl}
\Pi_{n,m}=\left\{\begin{array}{cl} \text{projection onto first } n \text{ rows}, & n<m, \\ \text{identity matrix } \eins_n, & n=m, \\ \text{inclusion into first } n \text{ rows}, & n>m. \end{array}\right.
\end{equation} 
Moreover we denote the Hermitian adjoint of a matrix $g$ by $g^*$ and the Vandermonde determinant of an $n\times n$ diagonal matrix $a$ by $\Delta_n(a)=\prod_{n\geq c>d\geq 1}(a_c-a_d)$. All these definitions help in constructing the following sets:
\begin{enumerate}
\item		the complex general linear group of dimension $n\times n$: $G_{n}={\rm Gl}_{\mathbb{C}}(n)$ equipped with the flat Lebesgue measure $dg$,
\item		the $n\times n$ invertible Hermitian matrices $H_{n}$ equipped with the flat Lebesgue measure $dx$,
\item		the  $n\times n$ real diagonal and invertible $n\times n$ matrices: $D_n\simeq\mathbb{R}^n$ equipped with the flat Lebesque measure $da$,
\item		the  $n\times n$ positive real diagonal $n\times n$ matrices $A_n\simeq\mathbb{R}_+^n$ equipped with the flat Lebesque measure $da$,
\item		the group of the complex lower  $n\times n$ triangular matrices
 			$$T_n=\biggl\{t\in G\biggl| t=\{t_{ab}\}_{a,b=1\ldots,n}\ {\rm with}\ t_{ij}\in\mathbb{C},\ t_{ij}=0\ {\rm for}\ i<j\biggl\}$$
			equipped with the flat Lebesgue measure $dt$,
\item		the unitary group $K_n={\rm U}(n)$ equipped with the normalized Haar measure $d^*k$,
\item		the complex $l\times m$ matrices of rank $n\leq l,m$
			$$G_{l,m}^{(n)}=\{k\Pi_{l,n}g\Pi_{n,m}\tilde{k}^*|k\in K_l,\ \tilde{k}\in K_m,\ {\rm and}\ g\in G_n\},$$
			where $G_{n,n}^{(n)}=G_n$ and equipped with the measure $dg'=d(k\Pi_{l,n}g\Pi_{n,m}\tilde{k}^*)=dg d^*kd^*\tilde{k}$ induced by those on $G_n$, $K_l$ and $K_m$,
\item		the $l\times l$ Hermitan matrices of rank $n\leq l$
			$$H_l^{(n)}=\{k\Pi_{l,n}x\Pi_{n,l}k^*|k\in K_l\ {\rm and}\ x\in H_n\},$$
			where $H_n^{(n)}=H_n$ and equipped with the measure $dx'=d(k\Pi_{l,n}x\Pi_{n,l}k^*)=dx d^*k$ induced by those on $H_n$ and $K_l$,
\item		and the symmetric group $\mathbb{S}_n$ permuting $n$ elements.
\end{enumerate}
The matrices of lower rank than their matrix dimensions are of particular interest since they allow to deal with products of the form $g_1g_2$ with $g_1\in \mathbb{C}^{l,n}$, $g_2\in \mathbb{C}^{n,m}$ and $l,m>n$ as well as of the form $gxg^*$ with $g\in\mathbb{C}^{l,n}$, $x\in {\rm Herm}(n)$ and $l>n$. Obviously both products are of lower rank and can be considered as embeddings of lower dimensional matrices via multiplicative maps. We would like to underline that indeed all matrices of fixed lower rank can be represented as shown in the definition of the two sets $G_{l,m}^{(n)}$ and $H_l^{(n)}$. For example, for an $l\times l$ dimensional Hermitian matrix $x'$ of rank $n$,  one particular representation is by choosing for $x$ the non-zero eigenvalues of $x'$ on the diagonal and $k$ comprises the corresponding eigenvectors arranged as columns. The construction for $G_{l,m}^{(n)}$ is very similar.

Additionally, we need the set of $K$-invariant Lebesgue integrable functions on $G_{l,m}^{(n)}$ and $H_l^{(n)}$  and the symmetric integrable functions on $D_n$ and $A_n$. These are defined as
\begin{equation}\label{L1.def}
\begin{split}
L^{1,K}(G_{l,m}^{(n)})=&\{ f_G\in L^{1}(G_{l,m}^{(n)})|\ f_G(g)=f_G(k_1gk_2)\ {\rm for\ all}\ g\in G_{l,m}^{(n)},\ k_1\in K_l\ {\rm and}\ k_2\in K_m\},\\
L^{(1,K)}(H_l^{(n)})=&\{ f_H\in L^{1}(H_l^{(n)})|\ f_H(x)=f_H(kxk^*)\ {\rm for\ all}\ x\in H_l^{(n)}\ {\rm and}\ k\in K_l\},\\
L^{1,\mathbb{S}}(D_n)=&\{f_D\in L^{1}(D_n)|\ f_D(a)=f_D(\sigma a\sigma^T)\ {\rm for\ all}\ a\in D_n\ {\rm and}\ \sigma\in \mathbb{S}_n\},\\
L^{1,\mathbb{S}}(A_n)=&\{f_A\in L^{1}(A_n)|\ f_A(a)=f_A(\sigma a\sigma^T)\ {\rm for\ all}\ a\in A_n\ {\rm and}\ \sigma\in \mathbb{S}_n\}.
\end{split}
\end{equation}
The matrix $\sigma^T$ is the transpose of $\sigma$. We coin the subset of probability densities of these function spaces by $L_{\rm Prob}^{1,K}(G_{l,m}^{(n)})$, $L_{\rm Prob}^{1,K}(H_l^{(n)})$ $L_{\rm Prob}^{1,\mathbb{S}}(D_n)$ and $L_{\rm Prob}^{1,\mathbb{S}}(A_n)$ and adapt the notation of~\cite{KK2019}. Additionally we emphasise the space on which a function belongs by a subscript, like $f_G$, $f_H$, $f_D$ or $f_A$.

The spaces $L^{1,K}(G_{l,m}^{(n)})$ and $L^{1,\mathbb{S}}(A_n)$ as well as $L^{1,K}(H_l^{(n)})$ and $L^{1,\mathbb{S}}(D_n)$ are bijectively related via the isometrics
\begin{equation}\label{IGdef}
\mathcal{I}_G:\ L^{1,K}(G_{l,m}^{(n)})\rightarrow L^{1,\mathbb{S}}(A_n),\ f_A(a)=\mathcal{I}_Gf_G(a)=\frac{\pi^{n^2}}{n!}\left(\prod_{j=0}^{n-1}\frac{1}{(j!)^2}\right)\Delta_n^2(a)f_G(\Pi_{l,n}a\Pi_{n,m}).
\end{equation}
and
\begin{equation}\label{IHdef}
\mathcal{I}_H:\ L^{1,K}(H_l^{(n)})\rightarrow L^{1,\mathbb{S}}(D_n),\ f_D(a)=\mathcal{I}_Hf_H(a)=\frac{1}{n!}\left(\prod_{j=0}^{n-1}\frac{\pi^j}{j!}\right)\Delta_n^2(a)f_H(\Pi_{l,n}a\Pi_{n,l}).
\end{equation}
The bijectivity follows from the fact that the normalized Haar measure of the group $K_n$ is unique. Hence, we can go back and forth in the matrix spaces without losing any information. 

\begin{remark}[Functions on the Complex Matrices $\mathbb{C}^{l\times m}$]\label{rem:complex}\

What might look strange is Eq.~\eqref{IGdef} since for the rectangular case $l\neq m$ we would expect an additional determinant. The reason why we do not have any here comes from the reference measure which is for a matrix $g'=k\Pi_{l,n} g\Pi_{n,m}\tilde{k}^* \in G_{l,m}^{(n)}$ with $g\in G_n$ given by the flat Lebesgue measure $dg$ (products of all real independent differentials) and $k\in K_l$ and $\tilde{k}\in K_m$ Haar distributed so that we have $dg'=dg d^*k d^*\tilde{k}$.

What is then the relation to $K$-invariant $L^1$-functions on $\mathbb{C}^{l\times m}$ that are distributed by the flat Lebesgue measure, since it is almost the set $G_{l,m}^{(n)}$ with $n=\min\{l,m\}$? First and foremost, all matrices of lower rank are of measure zero for this choice. Thus, we are compelled to look for another measure. In the case of a fully ranked matrix, say of rank $m$ with $l\geq m$, we choose a function $f_{\mathbb{C}}\in L^{1,K}(\mathbb{C}^{l\times m})$, in particular $f_{\mathbb{C}}(k\tilde{g}\tilde{k})=f_{\mathbb{C}}(\tilde{g})$ for all $\tilde{g}\in\mathbb{C}^{l\times m}$, $k\in K_{l}$ and $\tilde{k}\in K_m$. Then, the computation
\begin{equation}
\begin{split}
 \int_{\mathbb{C}^{l\times m}} d\tilde{g} f_{\mathbb{C}}(\tilde{g})=&\frac{\pi^{l m)}}{m!}\left(\prod_{j=0}^{m-1}\frac{1}{j!(j+l-m)!}\right)\int_{A_m}\Delta_m^2(a) \det a^{l-m} f_{\mathbb{C}}(\Pi_{l,m}a)\\
 =&\left(\prod_{j=0}^{m-1}\frac{\pi^{l-m}j!}{(j+l-m)!}\right)\int_{G_{l,m}^{(m)}}dg\int_{K_l}d^*k\int_{K_m}d^*\tilde{k} [\det (k\Pi_{l,m} g\tilde{k}^*)^*(k\Pi_{l,m} g\tilde{k}^*)]^{l-m}f_{\mathbb{C}}((k\Pi_{l,m} g\tilde{k}^*))
\end{split}
\end{equation}
leads to the identification of the corresponding $K$-invariant function on $G_{l,m}^{(m)}$ as
\begin{equation}\label{identification}
f_G(g')=\left(\prod_{j=0}^{m-1}\frac{\pi^{l-m}j!}{(j+l-m)!}\right)\det (g'{g'}^*)^{l-m}f_{\mathbb{C}}(g').
\end{equation}
Therefore, the well-known determinant is already part of $f_G$ and, hence, also comprised by $f_A$.
\end{remark}

The main focus of the present work is the multiplicative convolution
\begin{equation}\label{conv.H}
P_G\circledast Q_H(x)=\int_{G_{l,m}^{(n_1)}}d\tilde{g}\int_{H_m^{(n_2)}}d\tilde{x}\, \delta(x-\tilde{g}\tilde{x}\tilde{g}^*)\,P_G(\tilde{g})Q_H(\tilde{x})
\end{equation}
for $P_G$ an $L^1$-function on $G_{l,m}^{(n_1)}$ and $Q_H$ an  $L^1$-function on $H_m^{(n_2)}$, both being $K$-invariant. The Dirac delta function on $H_m^{(n_2)}$ is the one with respect to the measure $dx=d(k\Pi_{l,r}\hat{x}\Pi_{r,l} k^*)=d\hat{x} d^*k$ with $\hat{x}\in H_r$, $k\in K_l$ and $r=\min\{n_1,n_2\}$, i.e., $\int_{H_{l}^{(r)}}f(x)\delta(x-y)dx=f(y)$. It
can be explicitly evaluated for $l=m=n_1=n_2=n$,
\begin{equation}\label{conv.H-square}
P_G\circledast Q_H(x)=\int_{G_{n}}\frac{d\tilde{g}}{(\det \tilde{g}\tilde{g}^*)^{n}}P_G(\tilde{g})Q_H(\tilde{g}^{-1}x(\tilde{g}^{-1})^*).
\end{equation}
Its one dimensional counterpart is the well-known Mellin convolution
 \begin{equation}\label{uni-conv} 
 p\circledast q(a)= \int_{0}^\infty\frac{da'}{a'}p(a')q(a/a')
 \end{equation}
 for $p\in L^1(\mathbb{R}_+)$ and $q\in L^1(\mathbb{R})$.
 
The convolution~\eqref{conv.H} is closely related to multiplicative convolutions on the complex rectangular matrices,
\begin{equation}\label{conv.G}
P_G\circledast Q_G(g)=\int_{G_{l,m}^{(n_1)}}d\tilde{g}_1\int_{G_{m,o}^{(n_2)}}d\tilde{g}_2\,\delta(g-\tilde{g}_1\tilde{g}_2)\,P_G(\tilde{g}_1)Q_G(\tilde{g}_2)
\end{equation}
for two $K$-invariant $L^{1}$-functions on $G_{l,m}^{(n_1)}$ and $G_{m,o}^{(n_2)}$, respectively. Again, the Dirac delta function is the one with respect to the measure $dg=d(k\Pi_{l,r}\hat{g}\Pi_{r,o} \tilde{k}^*)=d\hat{g}d^*kd^*\tilde{k}$ with $\hat{g}\in G_r$, $k\in K_l$, $\tilde{k}\in K_o$ and the rank $r=\min\{n_1,n_2\}$.
For the group $G_n$, which is the case $l=m=o=n_1=n_2=n$, this evaluates to the well-known convolution
\begin{equation}\label{conv.G-square}
P_G\circledast Q_G(g)=\int_{G_n}\frac{d\tilde{g}}{(\det \tilde{g}\tilde{g}^*)^n}P_G(\tilde{g})Q_G(g\tilde{g}^{-1}).
\end{equation}
The latter convolution has been studied intensively in the literature~\cite{Helgason2000}. In random matrix theory, Eq.~\eqref{conv.G} has recently excited interests in applications in wireless telecommunications (e.g., see~\cite{TV2004,AKW2013,WZCT2015,KAAC2019}), quantum information~\cite{RSZ2011,Lakshminarayan2013,CHN2017}, and machine learning~\cite{HN2018,LC2018,TWJTN2019}. This culminated in rapid developments~\cite{AI2015} and has been connected to harmonic analysis~\cite{KK2016,KK2019,FKK2017}. Exactly the harmonic analysis approach is the path will pursue here, too.

To this aim, we need the Mellin transform on the full real line and not only on the half line. It can be considered as the direct sum of the Mellin transform of functions on the positive and negative real line, i.e.,
\begin{equation}\label{Mellin}
\begin{split}
\mathcal{M}: L^1(\mathbb{R})\rightarrow \mathcal{M} L^1(\mathbb{R});\qquad f\mapsto\mathcal{M}f(s,L)=\int_\mathbb{R}\frac{dx}{|x|}[\sign(x)]^L |x|^{s}f(x)
\end{split}
\end{equation}
for any $s\in\mathbb{C}$ and $L\in\{0,1\}$ where the integrand is Lebesgue integrable.  The function ${\rm sign}(x)$ is the signum function which yields the sign of a real number $x$ and vanishes if $x=0$.
Indeed, any function $f\in L^1(\mathbb{R})$ can be decomposed into $f(x)=f_+(x)+f_-(-x)$ with $f_{\pm}(\pm x)=f(x)\Theta(\pm x)$ where $\Theta$ is the Heaviside step-function. Then the ordinary Mellin transform is related to the one on the real line as $\mathcal{M}f(s,L)=\mathcal{M}f_+(s)+(-1)^L\mathcal{M}f_-(s)$. Due to the two copies of the complex plane $s\in\mathbb{C}$ denoted by $L\in\mathbb{Z}_2$, the Mellin transform~\eqref{Mellin} is still invertible and has the inverse
\begin{equation}\label{inv.Mellin}
\mathcal{M}^{-1}[\mathcal{M}f](x)=\lim_{\epsilon\to0}\sum_{L=0,1}\frac{[\sign(x)]^L}{4\pi}\int_{-\infty}^\infty ds \mathcal{M}f(\imath s+1,L)|x|^{-\imath s-1}\zeta_1(\epsilon s).
\end{equation}
The regularization
\begin{equation}
\zeta_1(\epsilon s)=\frac{\pi^2\cos(\epsilon s)}{\pi^2-4\epsilon^2 s^2}
\end{equation}
 guarantees that the  inverse holds for any $L^1$-function on the real line, and it can be omitted when the other terms in the integral~\eqref{inv.Mellin} are absolutely integrable. The sum over $L$ properly combines the two pieces $f_+$ and $f_-$ of a function $f$.
 
 The univariate Mellin transform decouples the multiplicative univariate convolution~\eqref{uni-conv}, i.e.,
 \begin{equation}\label{Mellin.fact}
 \mathcal{M}(p\circledast q)(s,L)=\mathcal{M} p(s) \mathcal{M} q(s,L)
 \end{equation}
 for $p\in L^1(\mathbb{R}_+)$ and $q\in L^1(\mathbb{R})$. We will construct the counterpart in matrix space.

In the next section, we generalize the construction~\eqref{Mellin} to the multivariate case of Hermitian matrices.

\section{Spherical Transforms}\label{sec:spherical}

The theory of spherical transforms~\cite{Helgason2000} developed by Harish-Chandra et al. in the 50's has been extremely helpful in dealing with sums~\cite{KR2016,FKK2017} and products~\cite{KK2016,KK2019,KFI2019} of random matrices. However, already when dealing with products of real antisymmetric and real asymmetric matrices, see~\cite{KFI2019}, we have seen that the original definition of the spherical transform does not carry far and needs to be modified. This is also here the case. The first step to achieve this goal is the definition of the spherical functions, which serve as the Fourier factors in ordinary Fourier analysis. To this goal, we define the two diagonal matrices
\begin{equation}\label{def.L.s}
s^{(n)}=\diag(n-1,n-2,\ldots,2,1,0)\in\mathbb{C}^n\qquad{\rm and}\qquad L^{(n)}=\diag({\rm mod}_2(n-1),{\rm mod}_2(n-2),\ldots,0,1,0)\in\mathbb{Z}_2^n.
\end{equation}

\begin{definition}[Spherical Functions on $H_{l}^{(n)}$ and $G_{l,m}^{(n)}$]\label{def:spherical-functions}\

Let $x\in H_{l}^{(n)}$ and $g\in G_{l,m}^{(n)}$ be two fixed matrices with the three positive integers $n\leq l,m$. Moreover, we set $s_{n+1}=-1$ and $L_{n+1}=-1$ and specify the rectangular matrix $\Pi_{j,l}$ as in Eq.~\eqref{proj-incl}. We define
\begin{enumerate}
\item		the spherical function on $H_{l}^{(n)}$ by
			\begin{equation}\label{spher:as}
				\Phi(s,L; x)=\frac{\int_{K_l}d^*k\prod_{j=1}^{n} {\rm sign}[\det \Pi_{j,l}kxk^*\Pi_{l,j}]^{L_j-L_{j+1}-1} |\det \Pi_{j,l}kxk^*\Pi_{l,j}|^{s_j-s_{j+1}-1}}{\int_{K_l}d^*k\prod_{j=1}^{n} |\det \Pi_{j,l}k\Pi_{l,n}\Pi_{n,l}k^*\Pi_{l,j}|^{s_j-s_{j+1}-1}}
			\end{equation}
			for $s=\diag(s_1,\ldots,s_n)\in\mathbb{C}^n$ and $L=\diag(L_1,\ldots,L_n)\in\mathbb{Z}_2^n$ with ${\rm Re}\,(s_j-s_{j+1})\geq 1$ for all $j=1,\ldots,n$ and analytically continue $\Phi$ to ${\rm Re}\,(s_b-s_{b+1})< 1$ for some $b=1,\ldots, n$,
\item		and the spherical function on $G_{l,m}^{(n)}$, cf., Ref.~\cite{Helgason2000,KK2016} for $G_n$,
			\begin{equation}\label{spher:s}
				\Psi(s; g)=\Phi(s,L^{(n)}; gg^*)
			\end{equation}
			for all $s\in\mathbb{C}^n$ and fixed $L^{(n)}$ as in Eq.~\eqref{def.L.s}.
\end{enumerate}
\end{definition}

We would like to underline that we employ here a different convention of $s_{n+1}$ compared to the standard literature~\cite{Helgason} where its value is usually $-(n+1)/2$. We have decided for the choice $s_{n+1}=-1$ so that the notation becomes simpler. One particular consequence is that the choice of $s$ to find the normalization of the spherical functions, which is $\Phi(s^{(n)},L^{(n)}; x)=1$ and $\Psi(s^{(r)}; g)=1$. Moreover, we have the trivial normalizations $\Phi(s,L; \Pi_{l,n}\Pi_{n,l})=1$ and $\Psi(s; \Pi_{l,n}\Pi_{n,m})=1$; the latter is due to $\Pi_{n,m}\Pi_{m,n}=\eins_n$. The normalising denominator in~\eqref{spher:as}, which we denote by
\begin{equation}\label{normalization}
C_{l,n}(s)=\int_{K_l}d^*k\prod_{j=1}^{n}|\det \Pi_{j,l}k\Pi_{l,n}\Pi_{n,l}k^*\Pi_{l,j}|^{s_j-s_{j+1}-1},
\end{equation}
is unity in the case of maximal rank, i.e., $l=n$.

\begin{remark}[Limit to Lower Ranked Matrices]\label{rem:limit}\

Evidently, the spherical functions $\Psi$ and $\Phi$ are $K$-invariant so that we can also study these function on $A_n$ and $D_n$, respectively. The relation between the spherical function on $G_{l,m}^{(n)}$ to the one on $G_{l,m}^{(n-1)}$ is given as follows
\begin{equation}
\begin{split}
&\lim_{a_n\to0} \lim_{s_n\to 0}C_{m}^{(n)}(\diag(s_1,\ldots,s_n))\Psi(\diag(s_1,\ldots,s_n); \diag(a_1,\ldots,a_n))\\
=&C_{m}^{(n-1)}(\diag(s_1,\ldots,s_{n-1}))\Psi(\diag(s_1,\ldots,s_{n-1})-\eins_{n-1}; \diag(a_1,\ldots,a_{n-1})).
\end{split}
\end{equation}
The order of the limits is of paramount importance otherwise it is zero or infinity, depending on the exponent, due to the vanishing determinant $\det a$ in the definition~\eqref{spher:as}. Similarly for the Hermitian matrices we find
\begin{equation}
\begin{split}
&\lim_{a_n\to0} \lim_{s_n, L_n\to 0}C_{m}^{(n)}(\diag(s_1,\ldots,s_n))\Phi(\diag(s_1,\ldots,s_n),\diag(L_1,\ldots,L_n); \diag(a_1,\ldots,a_n))\\
=&C_{m}^{(n-1)}(\diag(s_1,\ldots,s_{n-1}))\Phi(\diag(s_1,\ldots,s_{n-1})-\eins_{n-1},\diag(L_1,\ldots,L_{n-1})-\eins_{n-1}; \diag(a_1,\ldots,a_{n-1})).
\end{split}
\end{equation}
Relations with even lower ranks can be computed recursively. It is always crucial to multiply the normalization to the spherical function to find a smooth transition from higher to lower ranked matrices.
\end{remark}

The spherical function $\Psi$ for $G_n$ has a remarkable explicit analytical representation in terms of the squared singular values $a\in A_n$ of $g\in G_n$ given by the Gelfand-Na\u{\i}mark integral~\cite{GN1957}
\begin{equation}\label{GN-integral}
\Psi(s; g)=\Phi(s,L^{(n)}; a)=\left(\prod_{j=0}^{n-1}j!\right)\frac{\det[a_c^{s_b}]_{b,c=1,\ldots,n}}{\Delta_n(a)\Delta_n(s)}.
\end{equation}
The corresponding formula for general rectangular matrices $g\in G_{l,m}^{(n)}$ and Hermitian matrices $x\in H_{l}^{(n)}$ is very similar.

\begin{theorem}[Spherical Functions $\Phi$ and $\Psi$]\label{thm:spher.func}\

\begin{enumerate}
\item		Let $a\in D_n$ and $s\in\mathbb{C}^n$ with non-degenerate spectra, i.e. $a_l\neq a_k$ and $s_l\neq s_k$ for all $l\neq k$,  and $L\in\mathbb{Z}_2^n$. The spherical function~\eqref{spher:as} has the explicit form
			\begin{equation}\label{spher:as.b}
				\Phi(s,L; a)=\left(\prod_{j=0}^{n-1}j!\right)\frac{\det[[{\rm sign}(a_c)]^{L_b}|a_c|^{s_b}]_{b,c=1,\ldots,n}}{\Delta_n(a)\Delta_n(s)}.
			\end{equation}
\item			Let $a\in A_n$ and $s\in\mathbb{C}^n$ with non-degenerate spectra. The spherical function~\eqref{spher:s} is
			\begin{equation}\label{spher:s.b}
				\Psi(s; a)=\left(\prod_{j=0}^{n-1}j!\right)\frac{\det[a_c^{s_b}]_{b,c=1,\ldots,n}}{\Delta_n(a)\Delta_n(s)}.
			\end{equation}	
\item		The normalization~\eqref{normalization} is equal to
			\begin{equation}\label{normalization.b}
				C_{l,n}(s)=\prod_{j=1}^n\frac{(l-j)!\Gamma[s_j+1]}{(n-j)!\Gamma[s_j+l-n+1]}
			\end{equation}
			for any $s\in\mathbb{C}^{n}$.
\end{enumerate}

\end{theorem}

This theorem is proven in Appendix~\ref{Proof:spher.func} in a very similar way as the real counterpart with even dimensional antisymmetric matrices in~\cite{KFI2019}. This theorem shows that we can deal with all products of the form $gxg^*$ with $x$ being $m\times m$ Hermitian and $g$ a complex $n\times m$ rectangular with $m\leq n$ in a unified way. The following Proposition~\ref{prop:fact.spher} highlights this. Moreover, the case $m>n$ in the above product  can be considered, too, which is essentially a projection.

In the next step, we need to show a factorization theorem for the spherical transforms to be applicable to the convolution~\eqref{conv.H}.  For the spherical function~\eqref{spher:s} of the group $G_n$, this particular factorization reads~\cite{Helgason2000,KK2016}
\begin{equation}\label{factorization.square}
\int_{K_{n}}d^* k \Psi(s; gkg')=\Psi(s; g)\,\Psi(s;g')
\end{equation}
for any two $g,g'\in G_n$. Something similar is also true for the spherical function~\eqref{spher:as} as well as for the rectangular case which is proven in Appendix~\ref{Proof:fact.spher}.

\begin{proposition}[Factorization of $\Phi$ and $\Psi$]\label{prop:fact.spher}\

\begin{enumerate}
\item		Let $g\in G_{l,m}^{(n_1)}$ and $x\in H_m^{(n_2)}$ with $r=\min\{n_1,n_2\}$ the rank of $gxg^*$. Additionally, we choose $\tilde{s}\in\mathbb{C}^r$ and $\tilde{L}\in\{0,1\}^r$. Then, we find the following factorization
\begin{equation}\label{factorization}
\begin{split}
\int_{K_m}d^*k \Phi(s,L; gkxk^*g^*)=C_{m,r+|n_1-n_2|}(\tilde{s})\left\{\begin{array}{cl} \Psi(s; g)\ \Phi(\tilde{s},\tilde{L}; x), & n_1=r,\\ \Psi(\tilde{s}; g)\Phi(s,L; x), & n_2=r \end{array}\right.
\end{split}
\end{equation}
for any $s\in\mathbb{C}^r$ and $L\in\mathbb{Z}_2^r$. We defined $\tilde{s}=\diag(s+|n_1-n_2|\eins_{r},s^{(|n_1-n_2|)})$ and $\tilde{L}=\diag(L+|n_1-n_2)\eins_{r},L^{(|n_1-n_2|)})$.
\item		Let $g_1\in G_{l,m}^{(n_1)}$ and $g_2\in G_{m,o}^{(n_2)}$ with $r=\min\{n_1,n_2\}$ the rank of $g_1g_2$. Then, the factorization formula reads
\begin{equation}\label{factorization.rect}
\begin{split}
\int_{K_m}d^*k \Psi(s,L; g_1kg_2)=C_{m,r+|n_1-n_2|}(\tilde{s})\left\{\begin{array}{cl} \Psi(s; g_1)\Psi(\tilde{s}; g_2), & n_1=r,\\ \Psi(\tilde{s}; g_1)\Psi(s; g_2), & n_2=r \end{array}\right.
\end{split}
\end{equation}
for all $s\in\mathbb{C}^r$, where $\tilde{s}$  is defined as before.
\end{enumerate}

\end{proposition}

The second part of Proposition~\ref{prop:fact.spher} is equivalent to the discussion in~\cite{IK2014} for the product of complex rectangular random matrices stating that the singular value statistics is the same as product of a product of square random matrices at the expense of a random projection and the multiplication with a power of a determinant. The projection is represented here by the constants on the right hand side of Eq.~\eqref{factorization.rect}, which are essentially spherical transforms of random projections, see below.  The additional determinant in the weight is reflected by the shift of $s$ in $\Psi$. The latter can be seen by the identity
\begin{equation}\label{shiftPsi}
\Psi(s+\mu\eins_n;g)=(\det gg^*)^\mu\Psi(s;g),
\end{equation}
for $g\in G_n=G_{n,n}^{(n)}$. A similar identity exists for the spherical transform on the Hermitian matrices $H_n=H_n^{(n)}$,
\begin{equation}\label{shiftPhi}
\Phi_{n}(s+\mu\eins_n,L+j\eins_n;x)=(\sign(\det x))^j |\det x|^\mu \Phi(s,L;x)
\end{equation}
for any $\mu\in\mathbb{C}$ and $j\in\mathbb{Z}$.

\begin{remark}[Relation between Spherical Functions of $g$ and $g^*$]\label{rem:rel.adj}\

Another thing, which is worth mentioning, is that $\Psi(s; g)=\Phi(s,L^{(n)}; gg^*)$ only depends on the singular values $a\in A_n$ of the matrix $g$ which are shared by its Hermitian adjoint $g^*$. Thence, one may ask what the relation between both of their spherical functions is. The answer follows from Proposition~\ref{prop:fact.spher} above in combination with Definition~\ref{def:spherical-functions}. Without restriction of generality we assume $l\geq m$, then we calculate
\begin{equation}\label{spher-rel}
\begin{split}
\Psi(s; g)=\Phi(s,L^{(n)}; gg^*)=&\Phi(s,L^{(n)}; \Pi_{l,n}a\Pi_{n,l})\\
=&\int_{K_m}d^*k\Phi(s,L^{(n)}; \Pi_{l,m}k\Pi_{m,n}a\Pi_{n,m}k^*\Pi_{m,l})\\
=&C_{m,m}(\tilde{s})\Psi(\tilde{s}; \Pi_{l,m}) \Phi(s,L^{(n)}; \Pi_{m,n}a\Pi_{n,m})=\Phi(s,L^{(n)}; g^*g)=\Psi(s; g^*)
\end{split}
\end{equation}
with $\tilde{s}=\diag(s+(m-n)\eins_{n},s^{(m-n)})$.
In the first and second line, we have exploited the invariance of the Haar measure under multiplying unitary matrices, especially we have introduced the unitary matrix $\diag(k,\eins_{l-m})\in K_l$ which commutes with the projector $\Pi_{l,m}$ like $\diag(k,1_{l-m})\Pi_{l,m}=\Pi_{l,m}k$ and can be Haar distributed, too. In the third line we employed $C_{m,m}(\tilde{s})=1$. The spherical function $\Psi$ drops out because it is normalized for the projection $\Pi_{l,m}$ which has maximal rank.

This result shows that the spherical transform is the same for $g$ and $g^*$. This is not immediately clear from the definition because we integrate over different groups. Indeed without the normalization in the denominator~\eqref{spher:as}, we would have found a difference between both spherical transforms.
\end{remark}

The full benefit of Propositions~\ref{prop:fact.spher} unfurls when combining the factorization with the spherical transform, which is the Fourier transform on curved symmetric spaces. They are defined as follows.

\begin{definition}[Spherical Transforms corresponding to $\Phi_n$ and $\Psi_n$]\label{def:spher.trans}\
 
 Let $Q_G\in L^{1,K}(G_{l,m}^{(n)})$,  $P_H\in L^{1,K}(H_m^{(n)})$, $q_A=\mathcal{I}_GQ_G$, and $p_D=\mathcal{I}_HP_H$, see Eqs.~\eqref{IGdef} and~\eqref{IHdef}.
\begin{enumerate}
\item		The spherical transform $\mathcal{S}_{\Phi}: L^{1,K}(H_m^{(n)})\rightarrow \mathcal{S}_{\Phi}(L^{1,K}(H_m^{(n)}))$ corresponding to $\Phi$ is defined as
			\begin{equation}\label{Strafo:as}
				\begin{split}
				\mathcal{S}_{\Phi}P_H(s,L)=&\int_{H_m^{(n)}} dx P_H(x)\Phi(s,L; x)\\
				=&\left(\prod_{j=0}^{n-1}j!\right)\int_{D_n}da\, p_D(a)\frac{\det[[{\rm sign}(a_c)]^{L_b}|a_c|^{s_b}]_{b,c=1,\ldots,n}}{\Delta_n(a)\Delta_n(s)}=\mathcal{S}_{\Phi}p_D(s,L)
				\end{split}
			\end{equation}
			for those $s\in\mathbb{C}^n$ for which the integral exists and for any $L\in\mathbb{Z}_2^n$.
\item		The spherical transform $\mathcal{S}_{\Psi}: L^{1,K}(G_{l,m}^{(n)})\rightarrow \mathcal{S}_{\Psi}(L^{1,K}(G_{l,,m}^{(n)}))$ corresponding to $\Psi$ is (see~\cite{Helgason2000,KK2016} for $G_n$)
			\begin{equation}\label{Strafo:s}
			\begin{split}
			\mathcal{S}_{\Psi}Q_G(s)=&\int_{G_{l,m}^{(n)}}dg Q_G(g)\Psi_n(s; g)=\left(\prod_{j=0}^{n-1}j!\right)\int_{A_n}da\, q_A(a)\frac{\det[a_c^{s_b}]_{b,c=1,\ldots,n}}{\Delta_n(a)\Delta_n(s)}=\mathcal{S}_{\Psi}q_A(s)
			\end{split}
			\end{equation}
			for any $s\in\mathbb{C}^n$ where the integral exists.
\end{enumerate} 
In the second equalities, we slightly abuse notation and have to assume that $s_l\neq s_k$ for $l\neq k$.
\end{definition}

The spherical transforms are normalized such that
\begin{equation}
\mathcal{S}_{\Phi}P_H(s^{(n)},L^{(n)})=\int_{H_m^{(n)}} dx\,P_H(x)\qquad{\rm and}\qquad
\mathcal{S}_{\Psi}Q_G(s^{(n)})=\int_{G_{l,m}^{(n)}} dg\,Q_G(g).
\end{equation}
These relations are extremely helpful when fixing the constants for explicit ensembles.

\begin{example}[Spherical Transform of a Sub-Block of a Unitary Matrix]\label{ex:uni}\

As a simple example we would like to compute the spherical transform of a $l\times m$ sub-block matrix $g=\Pi_{l,M}k\Pi_{M,m}$ of a Haar distributed unitary matrix $k\in K_M$ with $l,m\leq M$. The probability density is then given by $P_G(g)=\int_{K_M}d^*k\delta(g-\Pi_{l,M}k\Pi_{M,m})$ with the Dirac delta function in the space $G_{l,m}^{(n)}$ with $n=\min\{l,m\}$. Without restriction of generality due to Eq.~\eqref{spher-rel}, we assume that $l\leq m$. Then we have
\begin{equation}\label{spher.uni}
\begin{split}
\int_{K_M} d^*k \Psi(s;\Pi_{l,M}k\Pi_{M,m})=C_{M,m}(\diag(s+(m-l)\eins_l,s^{(m-l)}))=\prod_{j=1}^l\frac{(M-j)!\Gamma[s_j+m-l+1]}{(m-j)!\Gamma[s_j+M-l+1]}
\end{split}
\end{equation}
because of Eq.~\eqref{factorization.rect} and the normalization of the spherical functions for orthogonal projections. Let us underline that we have not needed any restriction of $l$ and $m$ with respect to $M$ as it is has been the case in~\cite{KKS2016}. This is one of the strengths of the harmonic analysis approach presented here.
\end{example}

A direct consequence of Definition~\ref{def:spher.trans} in combination with Proposition~\ref{prop:fact.spher} is the factorization of the spherical transform when employing it on the convolutions~\eqref{conv.H} and~\eqref{conv.G}. For the convolution~\eqref{conv.G-square} on the square complex matrices $G_{n}$  this is well-known~\cite{Helgason2000,KK2016,KK2019} and explicitly reads
\begin{equation}
\mathcal{S}_{\Psi}[P_G\circledast Q_G](s)=\mathcal{S}_{\Psi}P_G(s)\mathcal{S}_{\Psi}Q_G(s)
\end{equation}
for two functions $P_G,Q_G\in L^{1,K}(G_{n})$.
The following corollary of Definition~\ref{def:spher.trans} and Proposition~\ref{prop:fact.spher} extends this result.

\begin{corollary}[Factorization Formulas of $S_\Phi$ and $S_\Psi$]\label{cor:fact:trans}\

\begin{enumerate}
\item		Let $P_G\in L^{1,K}(G_{l,m}^{(n_1)})$ and $Q_H\in L^{1,K}(H_m^{(n_2)})$ with $r=\min\{n_1,n_2\}$. Then the spherical transform of the convolution $P_G\circledast Q_H$ is
\begin{equation}\label{factorization.conv}
\begin{split}
\mathcal{S}_\Phi[P_G\circledast Q_H](s,L)=C_{m,r+|n_1-n_2|}(\tilde{s})\left\{\begin{array}{cl} \mathcal{S}_\Psi P_G(s)\mathcal{S}_\Phi Q_H(\tilde{s},\tilde{L}), & n_1=r,\\  \mathcal{S}_\Psi P_G(\tilde{s})\mathcal{S}_\Phi P_H(s,L), & n_2=r,\end{array}\right.
\end{split}
\end{equation}
			with $\tilde{s}=\diag(s+|n_1-n_2|\eins_{r},s^{(|n_1-n_2|)})$ and $\tilde{L}=\diag(L+|n_1-n_2|\eins_{r},L^{(|n_1-n_2|)})$.
\item		For $P_G\in L^{1,K}(G_{l,m}^{(n_1)})$ and $Q_G\in L^{1,K}(G_{m,o}^{(n_2)})$ with $r=\min\{n_1,n_2\}$, the spherical transform of the convolution $P_G\circledast Q_G$ becomes
\begin{equation}\label{factorization.rect.conv}
\begin{split}
\mathcal{S}_\Psi[P_G\circledast Q_G](s)=C_{m,r+|n_1-n_2|}(\tilde{s})\left\{\begin{array}{cl} \mathcal{S}_\Psi P_G(s)\mathcal{S}_\Psi Q_G(\tilde{s}), & n_1=r,\\ \mathcal{S}_\Psi P_G(\tilde{s})\mathcal{S}_\Psi Q_G(s), & n_2=r, \end{array}\right.
\end{split}
\end{equation}
			where $\tilde{s}=\diag(s+|n_1-n_2|\eins_{r},s^{(|n_1-n_2|)})$.
\end{enumerate}
\end{corollary}

\begin{proof}
The second part of this  corollary about the spherical transform on the complex matrices immediately results from the first due to the relation of the two spherical functions $\Phi$ and $\Psi$. Therefore, we only concentrate on the proof of the first part on the Hermitian matrices.

We consider the spherical transform
\begin{equation}\label{proof:fact:trans.1}
\begin{split}
\mathcal{S}_{\Phi}[P_G\circledast Q_H](s,L)=\int_{G_{l,m}^{(n_1)}}d g\int_{H_m^{(n_2)}} dx\Phi(s,L;gxg^*)P_G(g)Q_H(x),
\end{split}
\end{equation}
where we have evaluated the Dirac delta function in Eq.~\eqref{conv.H}. Now we use the fact that $P_G$ is $K$-invariant and introduce a Haar distributed unitary matrix $k\in K_m$ in the following way $g\to gk$. For the integration over $k$ we employ Proposition~\ref{prop:fact.spher} and find the claim~\eqref{factorization.conv}.
\end{proof}

As a last ingredient for solving the convolution~\eqref{conv.H} we need the invertibility of  the spehrical transform $S_\Phi$. When we have the inverse, the convolution can be rewritten as follows
\begin{equation}
P_G\circledast Q_H=\mathcal{S}_\Phi^{-1}[\mathcal{S}_{\Psi}P_G\,\mathcal{S}_{\Phi}Q_H].
\end{equation}
As we know from univariate probability theory, this representation is advantageous because the probabilistic averages can be readily carried out. We have proved the invertibility of the spherical transform in Appendix~\ref{Proof:inv} and the proper statement reads:

\begin{proposition}[Inverse of $S_\Phi$ and $S_\Psi$]\label{prop:inv}\

We define define the auxiliary function
\begin{equation}\label{xi.def}
\zeta_n(z)=\frac{\cos(z)}{\prod_{k=1}^n[1-4z^2/(\pi(2k-1))^2]}.
\end{equation}
\begin{enumerate}
\item		Let $P_H\in L^{1,K}(H_m^{(n)})$ and $p_D=\mathcal{I}_HP_H\in L^{1,\mathbb{S}}(D_n)$.
			The spherical transform $\mathcal{S}_{\Phi}$ is injective and, hence, invertible when restricted to its image. The inverse has the explicit form
			\begin{equation}\label{Strafo:as:inv}
			\begin{split}
			p_D(a)=\mathcal{S}_{\Phi}^{-1}[\mathcal{S}_{\Phi}p_D](a)=&\sum_{L\in\{0,1\}^n}\frac{\Delta_n(a)}{(n!)^2 \prod_{j=0}^{n-1}j!}\lim_{\epsilon\to0}\int_{\mathbb{R}^n} \frac{ds}{(4\pi)^{n}}\mathcal{S}_{\Phi}p_D(\imath s+s^{(n)},L)\prod_{l=1}^n\zeta_n(\epsilon s_l)\\
			&\times\Delta_n(\imath s+s^{(n)})\det[[{\rm sign}(a_c)]^{L_b}|a_c|^{-\imath s_b-n+b-1}]_{b,c=1,\ldots,n}
			\end{split}
			\end{equation}
			for almost all $a\in D_n$.
\item		Choosing $P_G\in L^{1,K}(G_{l,m}^{(n)})$ and $p_A=\mathcal{I}_GP_H\in L^{1,\mathbb{S}}(A_n)$, the spherical transform $\mathcal{S}_{\Psi}$ is invertible when restricted to its image with the inverse
			\begin{equation}\label{Strafo:s:inv}
			\begin{split}
			p_A(a)=\mathcal{S}_{\Psi}^{-1}[\mathcal{S}_{\Psi}p_A](a)=&\frac{\Delta_n(a)}{(n!)^2 \prod_{j=0}^{n-1}j!}\lim_{\epsilon\to0}\int_{\mathbb{R}^n} \frac{ds}{(2\pi)^{n}}\mathcal{S}_{\Psi}p_A(\imath s+s^{(n)})\prod_{l=1}^n\zeta_n(\epsilon s_l)\\
			&\times\Delta_n(\imath s+s^{(n)})\det[a_c^{-\imath s_b-n+b-1}]_{b,c=1,\ldots,n}
			\end{split}
			\end{equation}
			for almost all $a\in A_n$. For $G_n$ this statement is equal to the one in~\cite{KK2016}.
\end{enumerate}
\end{proposition}

As in the univariate case, the regularization $\zeta_n(z)$ is only needed when the the spherical transform is not $L^1$-integrable

Now, we have developed the theoretical framework to deal with products of random matrices on the Hermitian matrices. But before we come to that we want to underline that the above discussion has been true also for the whole set of $L^1$-functions which are $K$-invariant and not only for probability densities. In the ensuing discussion, when studying random matrices we restrict ourselves to the latter.

\section{Multiplicative Convolution on $G={\rm Gl}_{\mathbb{C}}(n)$ and $H={\rm Herm}(n)$}\label{sec:conv}

Already in previous works~\cite{KK2016,KK2019,KR2016,FKK2017,Kieburg2017,KFI2019}, it has been pointed out that particular ensembles are analytically easier to handle than others, in particular the integrals of the spherical transform and its inverse can be explicitly carried out. Here, one class of ensembles, called the polynomial ensembles~\cite{KZ2014}, are preferable to deal with since they always exhibit a determinantal point process even at finite matrix dimension~\cite{Borodin1998}. However, the convolution of two polynomials ensembles do not usually yield a polynomial ensembles. To overcome this obstacle a subclass has been identified~\cite{KK2016,KK2019,KR2016,FKK2017}. This subclass has been first coined the polynomial ensembles of derivative type~\cite{KK2016,KK2019,KR2016} due to its special form but then renamed to P\'olya ensembles~\cite{FKK2017}. The reason for the re-baptism is the fact that the ensembles are bijectively related to P\'olya frequency functions~\cite{Polya1913,Polya1915} that satisfy some differentiability and integrability conditions. The definition of these two sets requires two sets of $L^{1}$ functions on a subset $\mathbb{I}\subset\mathbb{R}$,
\begin{equation}\label{Ln1}
\begin{split}
L_n^{1}(\mathbb{I})=&\left\{f\in L_n^{1}(\mathbb{I})\left| \int_{\mathbb{I}}dx |x^j f(x)|<\infty\ {\rm for\ all}\ j\in[0,n]\right\}\right.
\end{split}
\end{equation}
and
\begin{equation}\label{Pn1}
\begin{split}
\mathcal{P}_n^{1}=&\left\{f\in L_n^{1}(\mathbb{R}_+)| f\circ\exp\ {\rm is}\ n\text{-times\ differentiable\ and\ a P\'olya\ frequency\ function\ of\ order}\ n, {\rm and}\right.\\
&\left. (x\partial)^jf\in L_n^{1}(\mathbb{R}_+)\  {\rm for\ all}\ j=0,\ldots,n\right\},
\end{split}
\end{equation}
where $f\circ\exp: \mathbb{R}\to\mathbb{R}$  is given by $f\circ\exp(x)=f(e^x)$.
We recall that a P\'olya frequency function is given as a real, non-zero function on $\mathbb{R}$ that satisfies the inequality~\cite{Polya1913,Polya1915}
\begin{equation}
\Delta_j(x)\Delta_j(y)\det[f(x_b-y_c)]_{b,c=1,\ldots,j}\geq0
\end{equation}
for all $x,y\in \mathbb{R}^j$ and $j=1,\ldots,n$.

\begin{definition}[Polynomial and P\'olya Ensembles~\cite{KZ2014,KK2016,KR2016,FKK2017}]\label{def:ensembles}\

\begin{enumerate}
\item		A polynomial ensemble on $G_{l,m}^{(n)}$ associated to the weights $w_1,\ldots,w_n\in L_{n-1}^1(\mathbb{R}_+)$ is a $K$-invariant ensemble whose squared singular values $a\in A$ are distributed as
			\begin{equation}\label{pol.G}
			p_A(a)=\frac{1}{n!}\Delta_n(a)\frac{\det[w_b(a_c)]_{b,c=1,\ldots,n}}{\det[\mathcal{M}w_b(c)]_{b,c=1,\ldots,n}}\in L_{\rm Prob}^{1}(A_n).
			\end{equation}
\item		A P\'olya ensemble on $G_{l,m}^{(n)}$ associated to the weight $\omega\in \mathcal{P}_{n-1}^{1}$ is a polynomial ensemble with
			\begin{equation}\label{Polya.G}
			w_b(a_c)=(-a_c\partial_c)^{b-1}\omega(a_c),
			\end{equation}
			i.e.,
			\begin{equation}\label{Polya.G.b}
			p_A(a)=\frac{1}{\prod_{j=1}^{n} j!\mathcal{M}\omega(j)}\Delta_n(a)\det[(-a_c\partial_c)^{b-1}\omega(a_c)]_{b,c=1,\ldots,n}.
			\end{equation}
\item		A polynomial ensemble on $H_m^{(n)}$ associated to the weights $w_1,\ldots,w_n\in L_{n-1}^1(\mathbb{R})$ is a $K$-invariant ensemble whose eigenvalues $a\in D$ are distributed as
			\begin{equation}\label{pol.H}
			p_D(a)=\frac{1}{n!}\Delta_n(a)\frac{\det[w_b(a_c)]_{b,c=1,\ldots,n}}{\det[\mathcal{M}w_b(c,c-1)]_{b,c=1,\ldots,n}}\in L_{\rm Prob}^{1}(D_n).
			\end{equation}
\end{enumerate}
\end{definition}

The true potential of the ensembles above unfolds itself when looking at their spherical transform that take particularly simple forms.

\begin{proposition}[Spherical Transforms of Polynomial and P\'olya Ensembles]\label{prop:spher.ensembles}\

\begin{enumerate}
\item		Let $P_G$ the distribution of a polynomial ensemble on $G_{l,m}^{(n)}$ associated to the weights $w_1,\ldots,w_n\in L_{n-1}^1(\mathbb{R}_+)$. Then, its spherical transform is (see~\cite{KK2016,KK2019} for the square case)
			\begin{equation}\label{spher.pol.G}
			\mathcal{S}_\Psi P_G(s)=\left(\prod_{j=0}^{n-1}j!\right)\frac{\det[\mathcal{M}w_b(s_c+1)]_{b,c=1,\ldots,n}}{\Delta_n(s)\det[\mathcal{M}w_b(c)]_{b,c=1,\ldots,n}},
			\end{equation}
			where $\mathcal{M}$ is the univariate Mellin transform on $\mathbb{R}_+$.
\item		The spherical transform of the distribution $P_G$, which is a P\'olya ensemble on $G_{l,m}^{(n)}$ associated to the weight $\omega\in \mathcal{P}_{n-1}^{1}$, is explicitly given by (see~\cite{KK2016,KK2019} for $l=m=n$)
			\begin{equation}\label{spher.Polya.G}
			\mathcal{S}_\Psi P_G(s)=\prod_{j=1}^{n}\frac{\mathcal{M}\omega(s_j+1)}{\mathcal{M}\omega(n-j+1)}.
			\end{equation}
\item		The spherical transform of  $P_H$ describing a polynomial ensemble on $H_m^{(n)}$ associated to the weights $w_1,\ldots,w_n\in L_{n-1}^1(\mathbb{R})$ is equal to
			\begin{equation}\label{spher.pol.H}
			\mathcal{S}_\Phi P_H(s,L)=\left(\prod_{j=0}^{n-1}j!\right)\frac{\det[\mathcal{M}w_b(s_c+1,L_c)]_{b,c=1,\ldots,n}}{\Delta_n(s)\det[\mathcal{M}\omega(c,c-1)]_{b,c=1,\ldots,n}}
			\end{equation}
			with $\mathcal{M}$ being the univariate Mellin transform on $\mathbb{R}$, see Eq.~\eqref{Mellin}.
\end{enumerate}
\end{proposition}

\begin{proof}
The first two parts are proven in~\cite{KK2016,KK2019} and the third claim is also straightforward. We only plug the definition~\eqref{pol.H} into Eq.~\eqref{Strafo:as}, apply Andr\'eief's identity~\cite{An86} and identify the integral in the determinant with the Mellin transform~\eqref{Mellin} on $\mathbb{R}$.
\end{proof}

An interesting consequence of the second part of Proposition~\ref{prop:spher.ensembles} is a relation between the spherical transforms of the P\'olya ensembles on $G_{l,m}^{(n)}$   associated to the weight $\omega\in \mathcal{P}_{n-1}^{1}$ and on $G_{l',m'}^{(n')}$ with the weight $\omega'(a)=a^{n-n'}\omega(a)$, where we assume $n>n'$. This relation reads
\begin{equation}
S_\Psi P_{G,\omega}(\diag(s+(n-n')\eins_{n-n'},s^{(n-n')})=S_\Psi P_{G,\omega'}(s).
\end{equation}
We will apply this observation several times in the ensuing sections.

It is a legitimate question to ask whether there is something like a P\'olya ensemble on $H_m^{(n)}$. Indeed, one can choose those ensembles induced by $x=\pm gg^*$ with $g\in G_{m,n}^{(n)}$ where $g$ is drawn from a P\'olya ensemble on $G_{m,n}^{(n)}$. However, one gets the feeling that these are the only realizations when looking for other ensembles. We have seen via trial and error that the positivity condition does not work well with the integrability when the support of the eigenvalues lies on both sides of the real axis. Since a proof is lacking let us phrase the above observation as a conjecture.

\begin{conjecture}[P\'olya ensembles on $H_m^{(n)}$]\label{conj:Polya}\

The only polynomial ensembles on $H_m^{(n)}$ that have the form $w_b(s_c)=(-s_c\partial_c)^{b-1}\omega(s_c)$ in Eq.~\eqref{pol.H} are those induced by $x=\pm gg^*$ with $g$ a random matrix drawn from a P\'olya ensemble on $G_{m,n}^{(n)}$.
\end{conjecture}

The reason why we care about P\'olya ensembles is the structural extremely simple form of their spherical transform, cf., Eq.~\eqref{spher.Polya.G}. It allows immediate conclusions for the convolutions~\eqref{conv.G} and~\eqref{conv.H}. For instance convolutions on $G_n$ of P\'olya ensembles are closed and have a semi-group action on polynomial ensembles on $G_n$, see~\cite{KK2019}. We will see that the P\'olya ensemble on $G_{l,m}^{(n_1)}$ have also a natural action on polynomial ensembles on $H_m^{(n_2)}$.

\begin{example}[Projections/Inclusions Revisited]\label{ex:proj}\

To generalize this to $G_{l,m}^{(n)}$ and $H_m^{(n)}$, we need to combine Proposition~\ref{prop:spher.ensembles} with the factorization Proposition~\ref{cor:fact:trans}. There the identification of the weight that creates the constant $C_{M,m}(\diag(s+(m-l)\eins_l,s^{(m-l)}))$ with $l\leq m\leq M$, see Eq.~\eqref{spher.uni}, is crucial. The good thing is that it corresponds to the weight
			\begin{equation}\label{omegal}
			\omega_{M-m}^{(m-l)}(a)=\left\{\begin{array}{cl} \displaystyle (M-m)a^{m-l}(1-a)^{M-m-1}\Theta(1-a), & M>m,\\ \delta(a-1), & M=m. \end{array}\right.
			\end{equation}
			The prefactor $(M-m)$ in the first case guarantees that the limit $M\to m$ yields the Dirac delta function in the second case  in the sense of weak topology. If $M\geq m+n$, this weight even satisfies the differentiability condition so that it can be associated to a P\'olya ensemble. This condition is exactly the one needed in~\cite{KKS2016} when considering products involving truncated unitary matrices.
\end{example}

Henceforth, we are more interested in the convolution~\eqref{conv.H} instead of Eq.~\eqref{conv.G}. The latter has been discussed extensively in several works over the past years~\cite{Kuijlaars2016,CKW2015,AI2015,KK2019,FKK2017}. Especially, we aim at computing the change of the bi-orthonormal functions $\{p_j,q_j\}_{j=0,\ldots,n-1}$ of the corresponding polynomial ensembles. Let us recall what these bi-orthonnormal functions are for a polynomial ensemble associated to the weights $\{w_j\}_{j=0,\ldots,n-1}$. The functions $\{p_j,q_j\}_{j=0,\ldots,n-1}$ satisfy three conditions, namely that $\{p_j\}_{j=0,\ldots,n-1}$ and $\{q_j\}_{j=0,\ldots,n-1}$  are bases of the linear span of the monomials $\{x^j\}_{j=0,\ldots,n-1}$ and of the weights $\{w_j\}_{j=1,\ldots,n}$, respectively, and additionally respect the bi-orthonormality condition
\begin{equation}\label{biortho}
\int da p_j(a)q_i(a)=\delta_{j,i}\ \text{for all}\ i,j=0,\ldots,n-1.
\end{equation}
Once these bi-orthonormal functions are given, it is straightforward to construct the kernel
\begin{equation}\label{kernel}
K_n(a_1,a_2)=\sum_{j=0}^{n-1}p_j(a_1)q_j(a_2)
\end{equation}
that determines all $\kappa$-point correlation functions
\begin{equation}\label{kpoint}
R_\kappa(a_1,\ldots a_{\kappa})=\frac{n!}{(n-\kappa)!}\prod_{j=\kappa+1}^n\int da_j p(a)=\det[K_n(a_b,a_c)]_{b,c=1,\ldots,\kappa}.
\end{equation}
The latter comprises all spectral statistical information of the random matrix. We would like to emphasize that the bi-orthonormal functions are not unique. The various choices can be exploited by picking the one suited the best for the considered problem. Additionally, one can specify a polynomial ensemble, either on $G_{l,m}^{(n)}$ or on $H_m^{(n)}$, by its bi-orthonormal functions $\{p_j,q_j\}_{j=0,\ldots,n-1}$ instead of the associated weights $\{w_j\}_{j=1,\ldots,n-1}$. We will exploit this when multiplying P\'olya ensembles on $G_{l,m}^{(n_1)}$ with polynomial ensembles on $H_m^{(n_2)}$ in Subsection~\ref{sec:mult.H}. Therein, the goal will be to understand the change of the bi-orthonormal functions.

\subsection{Multiplication of a Fixed Matrix on $H_m^{(n)}$}\label{sec:mult.G}

First, we would like to consider the case where $x\in H_m^{(n_2)}$ is fixed and $g\in G_{l,m}^{(n_1)}$ is a P\'olya ensemble.  Such a random matrix has the following joint probability densities of its eigenvalues.

\begin{theorem}[JPDF of a P\'olya Ensemble on $G_{l,m}^{(n_1)}$ Multiplied to a Fixed Matrix in $H_m^{(n_2)}$]\label{thm:jpdf.fixed}\

We choose four integers $l,m, n_1,n_2$ satisfying the two conditions $l,m\leq n_1$ and $m\leq n_2$. Moreover, we draw a random matrix $g\in G_{l,m}^{(n_1)}$ described by a P\'olya ensemble associated to the weight $\omega\in \mathcal{P}_{n-1}^{1}$ and $x\in H_m^{(n_2)}$ fixed with $a\in D_n$ its non-zero and non-degenerate eigenvalues. Then, we have the two cases:
\begin{enumerate}
\item		$n_1\geq n_2$:	Then the eigenvalues $\tilde{a}$ of the product $gag^*$ are distributed by
\begin{equation}\label{pfixed.g}
p(\tilde{a}|a)=\frac{1}{n_2!}\left(\prod_{j=1}^{n_2}\frac{(m-j)!}{(m-n_1)!(n_1-j)!\mathcal{M}\omega(n_1-j+1)}\right)\frac{\Delta_{n_2}(\tilde{a})}{\Delta_{n_2}(a)}\det\left[\tilde{\omega}_\geq(\tilde{a}_b|a_c)\right]_{b,c=1,\ldots,n_2}
\end{equation}
with the weight
\begin{equation}\label{weight.fixed.g}
\tilde{\omega}_\geq(\tilde{a}_b|a_c)=\frac{1}{|a_c|}\tilde{\omega}_\geq\left(\left.\frac{\tilde{a}_b}{a_c}\right|1\right)=\Theta(\tilde{a}_b a_c)\left(\frac{\tilde{a}_b}{a_c}\right)^{n_1-n_2}\int_0^\infty \frac{d a'}{|a_c|a'}\omega_{m-n_1}^{(0)}(a')\omega\left(\frac{\tilde{a}_b}{a_ca'}\right),
\end{equation}
where $\omega_{m-n_1}^{(0)}$ is defined in Eq.~\eqref{omegal} and $\Theta$ the Heaviside step function. The integral is evaluated at $a'=1$ when $m=n_1$.
\item		$n_1< n_2$: In this case the joint probability distribution of the eigenvalues $\tilde{a}$ of the product $gag^*$ is
\begin{equation}\label{pfixed.l}
\begin{split}
p(\tilde{a}|a)=&\frac{1}{n_1!}\left(\prod_{j=1}^{n_1}\frac{(m-j)!}{(m-n_1)!(n_1-j)!\mathcal{M}\omega(n_1-j+1)}\right)\frac{\Delta_{n_1}(\tilde{a})}{\Delta_{n_2}(a)}\det\left[\begin{array}{c} a_c^{b-1} \\ \tilde{\omega}_<(\tilde{a}_d|a_c) \end{array} \right]_{\substack{b=1,\ldots, n_2-n_1 \\ d=1,\ldots,n_1 \\ c=1,\ldots, n_2}}\\
=&\frac{1}{n_1!}\left(\prod_{j=1}^{n_1}\frac{(m-j)!}{(m-n_1)!(n_1-j)!\mathcal{M}\omega(n_1-j+1)}\right)\Delta_{n_1}(\tilde{a})\det\left[\sum_{c=1}^{n_2} \frac{e_{n_1-b}(-a_{\neq c})}{\prod_{h\neq c}(a_c-a_h)}\tilde{\omega}_<(\tilde{a}_d|a_c)\right]_{b,d=1,\ldots,n_1},
\end{split}
\end{equation}
with $a_{\neq j}=\diag(a_1,\ldots,a_{j-1},a_{j+1},\ldots,a_{n_2})\in D_{n_2-1}$, the weight
\begin{equation}\label{weight.fixed.l}
\tilde{\omega}_<(\tilde{a}_d|a_c)=\frac{a_c^{n_2-n_1}}{|a_c|}\tilde{\omega}_<\left(\left.\frac{\tilde{a}_d}{a_c}\right|1\right)=\Theta(\tilde{a}_d a_c)a_c^{n_2-n_1}\int_0^\infty \frac{d a'}{|a_c|a'}\omega_{m-n_1}^{(0)}(a')\omega\left(\frac{\tilde{a}_d}{a_ca'}\right)
\end{equation}
and the elementary symmetric polynomials of order $o\leq n_2-1$ of $n_2-1$ elements
\begin{equation}\label{elementary}
e_o(-a_{\neq j})=\oint\frac{dz}{2\pi i z^{n_2-o}} \prod_{h\neq j}(z-a_h),
\end{equation}
where the contour is taken counter-clockwise around the origin.
\end{enumerate}
In particular, in both cases the random matrix $g\tilde{a} g^*$ is equal to a polynomial ensemble on $H_l^{(\min\{n_1,n_2\})}$. For degenerate spectra of $x$ one needs to perform l'H\^opital's rule.
\end{theorem}

The joint probability density is the starting point in deriving the spectral statistics in terms of its kernel. For this purpose it is crucial to find the corresponding bi-orthonormal functions. The good thing is that we have already well-prepared the weight for that in Theorem~\ref{thm:jpdf.fixed}.

\begin{corollary}[Eigenvalue Statistics of Products of P\'olya Ensembles with Fixed Matrices]\label{cor:stat.fixed}\

We consider the same setting as in Theorem~\ref{thm:jpdf.fixed}. Additionally, we define
\begin{equation}\label{orth.fixed}
\begin{split}
				p_j(\tilde{a})=&a^j,\\
				q_j^{\geq}(\tilde{a})=&\frac{(m+j-n_2)!}{(m-n_1)!(n_1+j-n_2)!}\sum_{c=1}^{n_2} \frac{e_{n_2-j-1}(-a_{\neq c})}{\prod_{h\neq c}(a_c-a_h)}\frac{\tilde{\omega}_\geq(\tilde{a}|a_c)}{\mathcal{M}\omega(n_1+j-n_2+1)},\\
				q_j^{<}(\tilde{a})=&\frac{(m+j-n_1)!}{(m-n_1)!j!}\sum_{c=1}^{n_2} \frac{e_{n_1-j-1}(-a_{\neq c})}{\prod_{h\neq c}(a_c-a_h)}\frac{\tilde{\omega}_<(\tilde{a}|a_c)}{\mathcal{M}\omega(j+1)}.
\end{split}
\end{equation}
Then,  the functions $\{p_j,q_j^\geq \}_{j=0,\ldots,n_2-1}$  are the bi-orthonormal set for $n_1\geq n_2$ and $\{p_j,q_j^<\}_{j=0,\ldots,n_1-1}$ is it for $n_1< n_2$. Hence, the kernels are given as
\begin{equation}\label{kern.fixed.g}
K_{n_2}^{\geq}(\tilde{a}_1,\tilde{a}_2)=\sum_{j=0}^{n_2-1}\sum_{c=1}^{n_2}\frac{(m+j-n_2)!}{(m-n_1)!(n_1+j-n_2)!}\frac{e_{n_2-j-1}(-a_{\neq c})}{\prod_{h\neq c}(a_c-a_h)}\frac{\tilde{a}_1^{j-1}\ \tilde{\omega}_\geq(\tilde{a}_2|a_c)}{\mathcal{M}\omega(n_1+j-n_2+1)}
\end{equation}
for $n_1\geq n_2$ and
\begin{equation}\label{kern.fixed.l}
K_{n_1}^<(\tilde{a}_1,\tilde{a}_2)=\sum_{j=0}^{n_1-1}\sum_{c=1}^{n_2}\frac{(m+j-n_1)!}{(m-n_1)!j!}\frac{e_{n_1-j-1}(-a_{\neq c})}{\prod_{h\neq c}(a_c-a_h)}\frac{\tilde{a}_1^{j-1}\ \tilde{\omega}_<(\tilde{a}_2|a_c)}{\mathcal{M}\omega(j+1)}
\end{equation}
for $n_1< n_2$.
\end{corollary}

\begin{proof}
For $n_1\geq n_2$ the statement above immediately follows from the computation
\begin{equation}\label{proof.fixed.1}
\int_{-\infty}^\infty d\tilde{a}\ \tilde{a}^{b}q_{b'}^\geq(\tilde{a})=\sum_{c=1}^{n_2} \frac{e_{n_2-b'-1}(-a_{\neq c})}{\prod_{h\neq c}(a_c-a_h)} a_c^{b}=\delta_{bb'},
\end{equation}
where the last equality has been shown in Eq.~\eqref{elem.orth}. Moreover, the functions $q_j$ are obviously a linear combination of the original weights $\{\tilde{\omega}_\geq(\tilde{a}|a_c)\}_{c=1,\ldots,n_2}$, which is everything to be shown. In exactly the same way one shows the claim for $n_1<n_2$.
\end{proof}

It is interesting that the structural form of the statistics for the embedding of a matrix from $x$ to $gxg^*$ with $n_1>n_2$ is not that much different from a projection ($n_1<n_2$). This would hint to an intimate relation between both operations albeit the statistics are certainly different. While for $n_1>n_2$ the number of non-zero eigenvalues stays the same, the rank of the matrix  decreases for $n_1<n_2$.

Another point we would like to highlight is  the representation of the kernels~\eqref{kern.fixed.g} and~\eqref{kern.fixed.l} as sums that can also be written in contour integrals under similar conditions as done in~\cite{Kieburg2017} for additive convolutions on matrix spaces. We will omit them here and go over to choosing the matrix $x\in H_{m}^{(n_2)}$ also randomly. Indeed for the case of $g$ being a product of Ginibre matrices this has been achieved in~\cite{Liu2017}.

\subsection{Multiplication of a Polynomial Ensemble on $H_m^{(n)}$}\label{sec:mult.H}

As before the joint probability density will be the starting point of our analysis.

\begin{theorem}[JPDF of a P\'olya Ensemble on $G_{l,m}^{(n_1)}$ Multiplied to a Polynomial Ensemble on in $H_m^{(n_2)}$]\label{thm:jpdf.poly}\

We consider the situation in Theorem~\ref{thm:jpdf.fixed} except that $x\in H_m^{(n_2)}$ is drawn from a polynomial ensemble associated to the weights $w_1,\ldots,w_n\in L_{n-1}^1(\mathbb{R})$. Again, we have to do the following case discussion:
\begin{enumerate}
\item		$n_1\geq n_2$:	The eigenvalues $\tilde{a}$ of the random  matrix product $gxg^*$ are distributed by
\begin{equation}\label{ppoly.g}
p(\tilde{a})=\frac{1}{n_2!}\left(\prod_{j=1}^{n_2}\frac{(m-j)!}{(m-n_1)!(n_1-j)!\mathcal{M}\omega(n_1-j+1)}\right)\Delta_{n_2}(\tilde{a})\frac{\det\left[\int_{-\infty}^\infty da \tilde{\omega}_\geq(\tilde{a}_b|a)w_c(a)\right]_{b,c=1,\ldots,n_2}}{\det\left[\mathcal{M}w_b(c,c-1)\right]_{b,c=1,\ldots,n_2}}
\end{equation}
where $\tilde{\omega}_\geq$ is given in Eq.~\eqref{weight.fixed.g}.
\item		$n_1< n_2$: The joint probability distribution of the eigenvalues $\tilde{a}$ of $gag^*$ is
\begin{equation}\label{ppoly.l}
\begin{split}
p(\tilde{a}|a)=\frac{1}{n_1!}\left(\prod_{j=1}^{n_1}\frac{(m-j)!}{(m-n_1)!(n_1-j)!\mathcal{M}\omega(n_1-j+1)}\right)\Delta_{n_1}(\tilde{a})\frac{\det\left[\begin{array}{c} \mathcal{M}w_c(b,b-1) \\ \int_{-\infty}^\infty da\tilde{\omega}_<(\tilde{a}_d|a)w_c(a) \end{array} \right]_{\substack{b=1,\ldots, n_2-n_1 \\ d=1,\ldots,n_1 \\ c=1,\ldots, n_2}}}{\det\left[\mathcal{M}w_b(c,c-1)\right]_{b,c=1,\ldots,n_2}},
\end{split}
\end{equation}
with the weight $\tilde{\omega}_<$ from Eq.~\eqref{weight.fixed.l}.
\end{enumerate}
\end{theorem}

\begin{proof}
The proof is straightforward since we only need to multiply Eqs.~\eqref{pfixed.g} and~\eqref{pfixed.l} with the probability density~\eqref{Polya.G.b} and integrate over $a\in D_{n_2}$. The Vandermonde determinant $\Delta_{n_2}(a)$ cancels and the remaining integral is carried out with the aid of Andr\'eief's identity~\cite{An86} leading to the statements.
\end{proof}

The second result~\eqref{ppoly.l} for $n_1<n_2$ as well as the first one~\eqref{ppoly.g} can be simplified drastically when choosing a set of bi-orthonormal functions $\{p_j,q_j\}_{j=0,\ldots,n_2-1}$ for the polynomial ensemble of $x\in H_m^{(n_2)}$. Firstly, the determinant $\det\left[\mathcal{M}w_b(c,c-1)\right]_{b,c=1,\ldots,n_2}$ in the  denominators is replaced by $\det\left[\int_{-\infty}^\infty da\,p_{b-1}(a)q_{c-1}(a)\right]_{b,c=1,\ldots,n_2}=1$ which equals unity. Secondly, the determinant in the numerator of Eq.~\eqref{ppoly.l} is essentially of size $n_1\times n_1$ since the first $n_2-n_1$ rows become Kronecker deltas. Summarizing, in the terms of the bi-orthonormal functions of the random matrix $x$ , the two equations~\eqref{ppoly.g} and~\eqref{ppoly.l} are equal to
\begin{equation}\label{ppoly.g.b}
p(\tilde{a})=\frac{1}{n_2!}\left(\prod_{j=1}^{n_2}\frac{(m-j)!}{(m-n_1)!(n_1-j)!\mathcal{M}\omega(n_1-j+1)}\right)\Delta_{n_2}(\tilde{a})\det\left[\int_{-\infty}^\infty da \tilde{\omega}_\geq(\tilde{a}_b|a)q_{c-1}(a)\right]_{b,c=1,\ldots,n_2}
\end{equation}
and
\begin{equation}\label{ppoly.l.b}
p(\tilde{a}|a)=\frac{1}{n_1!}\left(\prod_{j=1}^{n_1}\frac{(m-j)!}{(m-n_1)!(n_1-j)!\mathcal{M}\omega(n_1-j+1)}\right)\Delta_{n_1}(\tilde{a})\det\left[\int_{-\infty}^\infty da\tilde{\omega}_<(\tilde{a}_d|a)q_{n_2-n_1+c-1}(a) \right]_{c,d=1,\ldots,n_1},
\end{equation}
respectively.

 This simplification comes in handy when  analysing the transformations of the statistics, especially the set of bi-orthonormal functions, of the Hermitian random matrix $x\in H_m^{(n_2)}$ to the one  of $gxg^*\in H_l^{(r)}$ with $r=\min\{n_1,n_2\}$. There are two reason why this might be of interest. The first one is a practical one. As already explained the bi-orthonormal functions build up the whole spectral statistics for this class of ensembles. The second reason aims at a better understanding of the random matrix multiplication as a statistical process. We have formulated the ``analytical response" of such a multiplication in the following proposition.

\begin{proposition}[Eigenvalue Statistics of Products of P\'olya Ensembles with Polynomial Ensembles]\label{prop:stat.poly}\

We consider the setting of Theorem~\ref{thm:jpdf.poly} and assume that the polynomial ensemble of $x$ corresponds to the bi-orthonormal functions $\{p_j,q_j\}_{j=0,\ldots,n_2-1}$ with the kernel $K_{n_2}(a_1,a_2)$.
Moreover, we define the polynomials
\begin{equation}\label{chi.def}
\begin{split}
\chi_\geq(z)=&\sum_{j=0}^{n_2-1}\frac{(m+j-n_2)!}{(m-n_1)!(n_1+j-n_2)!\mathcal{M}\omega(n_1+j-n_2+1)}z^j,\\
\chi_<(z)=&\sum_{j=0}^{n_1-1}\frac{(m+j-n_1)!}{(m-n_1)!j!\mathcal{M}\omega(j+1)}z^{j+n_2-n_1}.
\end{split}
\end{equation}
 Then, the bi-orthogonal functions $\{\tilde{p}_j,\tilde{q}_j\}_{j=0,\ldots,r-1}$ and kernel $\tilde{K}_r(\tilde{a}_1,\tilde{a}_2)$ of the product $gxg^*$ with rank $r=\min\{n_1,n_2\}$ are
\begin{enumerate}
\item		$n_1\geq n_2$: 
			\begin{equation}\label{stat.pol.g}
			\begin{split}
			\tilde{p}_j(\tilde{a})=&\oint \frac{dz}{2\pi i z}\chi_\geq(z)p_j\left(\frac{\tilde{a}}{z}\right),\\
			\tilde{q}_j(\tilde{a})=&\int_{-\infty}^\infty \frac{ da}{|a|}\tilde{\omega}_\geq\left(\left.\frac{\tilde{a}}{a}\right|1\right)q_j(a),\\
			\tilde{K}_{n_2}(\tilde{a}_1,\tilde{a}_2)=&\oint \frac{dz}{2\pi i z}\chi_\geq(z)\int_{0}^\infty  \frac{ da}{a}\tilde{\omega}_\geq\left(\left.a\right|1\right)K_{n_2}\left(\frac{\tilde{a}_1}{z},\frac{\tilde{a}_2}{a}\right),\\
			\end{split}
			\end{equation}
\item		$n_1<n_2$: 
			\begin{equation}\label{stat.pol.l}
			\begin{split}
			\tilde{p}_j(\tilde{a})=&\tilde{a}^{n_1-n_2} \oint \frac{dz}{2\pi i z}\chi_<(z)p_{n_2-n_1+j}\left(\frac{\tilde{a}}{z}\right),\\
			\tilde{q}_j(\tilde{a})=&\int_{-\infty}^\infty  \frac{da}{|a|}a^{n_2-n_1}\tilde{\omega}_<\left(\left.\frac{\tilde{a}}{a}\right|1\right),\\
			\tilde{K}_{n_1}(\tilde{a}_1,\tilde{a}_2)=&\left(\frac{\tilde{a}_2}{\tilde{a}_1}\right)^{n_2-n_1}\oint \frac{dz}{2\pi i z}\chi_<(z)\int_{0}^\infty \frac{da}{a^{n_2-n_1+1}}\tilde{\omega}_<\left(\left.a\right|1\right)K_{n_2}\left(\frac{\tilde{a}_1}{z},\frac{\tilde{a}_2}{a}\right).\\
			\end{split}
			\end{equation}
			The contour integrals run counter-clockwise around the origin.
\end{enumerate}
\end{proposition}

The upper and lower limits for $\chi_\geq$ and $\chi_<$ can be chosen arbitrarily if more moments exist. Indeed it can be even taken to infinity if the series exists in a ring around the origin. This is the case, for instance, for the Ginibre ensemble where $\mathcal{M}\omega(s)=\Gamma(s)$ and for the Jacobi ensemble with  $\mathcal{M}\omega(s)=\Gamma(s+\nu)\Gamma(\mu+n)/\Gamma(s+\nu+\mu+n)$ for $\nu>-1$ and $\mu> 0$, see~\cite{KKS2016,KK2019}. The case of the product of an induced Ginibre ensemble whose weight is equal to $\omega(a)=x^\nu e^{-a}$ has been studied in~\cite{FIL2018}; in particular Theorem~\ref{thm:jpdf.poly} and Proposition~\ref{prop:stat.poly}, up to normalization, become Lemma 2 and Proposition 7 in~\cite{FIL2018}, respectively. The statements above generalize this discussion and avoids the non-compact group integrals encountered in~\cite{FIL2018} that are unknown for more general weights, e.g., for the Jacobi ensemble with $\omega(a)=a^\nu(1-a)^{\mu+n-1}\Theta(1-a)$, Cauchy-Lorentz ensemble with $\omega(a)=a^\nu/(1+a)^{\mu+n}$ or the Muttalib--Borodin ensembles like $\omega(a)=a^\nu e^{-a^\theta}$ with $\theta>0$ or  $\omega(a)=a^\nu e^{-({\rm ln}\,a)^2}$. All of these examples are P\'olya ensembles, and our theoretical framework and its results deals with them in a unifying way.

\begin{example}[Projections and Inclusions of Hermitian Matrices]\label{ex:appl}\

We would like to conclude this section with an example  highlighting that $\omega$ can be even a distribution. This example is given by a projection or an inclusion of the Hermitian matrix $x\in H_{m}^{(n_2)}$.

For the orthogonal projection we choose $g\in G_{l,m}^{(l)}$ (in particular $l\leq m$) to be an $l\times m$ block of a Haar distributed unitary matrix of size $M\times M$, see Example~\ref{ex:uni}. Then, we know from Eq.~\eqref{spher.uni} comparing with Eq.~\eqref{spher.Polya.G} that the corresponding Mellin transform of $\omega$ is given by
\begin{equation}
\mathcal{M}\omega(s)=\frac{\Gamma[M-m+1]\Gamma[s+m-l]}{\Gamma[s+M-l]}.
\end{equation}
This corresponds to the identification $\omega=\omega_{M-m}^{(m-l)}$, see Example~\ref{ex:proj}, which does not always satisfy the differentiability criterion for arbitrary $m,l\leq M$ leading to a distribution for the joint probability density. Nonetheless, we can apply the factorization of the spherical transform, see Corollary~\ref{cor:fact:trans}. This leads to the two cases
\begin{equation}
\begin{split}
\tilde{\omega}_\geq(\tilde{a}|a)=&\frac{(m-l)!(M-m)!}{(M-l)!}\frac{\Theta(\tilde{a} a)}{|a|}\omega_{M-l}^{(l-n_2)}\left(\frac{\tilde{a}}{a}\right),\quad \chi_\geq(z)=\sum_{j=0}^{n_2-1}\frac{(M+j-n_2)!}{(m-l)!(l+j-n_2)!(M-m)!}z^j
\end{split}
\end{equation}
and
\begin{equation}
\begin{split}
\tilde{\omega}_<(\tilde{a}|a)=&\frac{(m-l)!(M-m)!}{(M-l)!}\frac{\Theta(\tilde{a} a)}{|a|}\omega_{M-l}^{(0)}\left(\frac{\tilde{a}}{a}\right),\quad \chi_<(z)=\sum_{j=0}^{l-1}\frac{(M-l+j)!}{(m-l)!j!(M-m)!}z^{j+n_2-l}
\end{split}
\end{equation}
depending on whether $l\geq n_2$ or $l< n_2$, respectively. An orthogonal projection in the original sense is given by $m=M$ since the unitary matrix can be absorbed into $x$ due to the $K$-invariance of the latter. Thus it reduces to a projection onto the first $l$ rows. For $l=m=M$ we notice that the functions naturally reduce to those that keep the statistics the same.

For the inclusion we consider $l\geq m$ with $g\in G_{l,m}^{(m)}$. This leads to the Mellin transform of $\omega$ 
\begin{equation}
\mathcal{M}\omega(s)=\frac{\Gamma[M-l+1]\Gamma[s+l-m]}{\Gamma[s+M-m]}.
\end{equation}
which corresponds to $\omega=\omega_{M-l}^{(l-m)}$.  Since there is only one situation to consider namely $l\geq m\geq n_2$, the constant in front of Corollary~\ref{cor:fact:trans} is equal to unity ($C_{m,m}(\tilde{s})=1$). Therefore the two functions in the transformation of the eigenvalue statistics are
\begin{equation}
\begin{split}
\tilde{\omega}_\geq(\tilde{a}|a)=&\frac{\Theta(\tilde{a} a)}{|a|}\omega_{M-l}^{(l-m)}\left(\frac{\tilde{a}}{a}\right),\quad \chi_\geq(z)=\sum_{j=0}^{n_2-1}\frac{(M+j-m)!}{(l+j-m)!(M-l)!}z^j.
\end{split}
\end{equation}
A true random inclusion, that is unitarily invariant, would be the case $M=l$ which interestingly does not change the statistics of $x$ to $gxg^*$ at all. Indeed it becomes immediately clear why it is so when considering the characteristic polynomial where we have
\begin{equation}
\det(gxg^*-a\eins_M)=\det(\Pi_{M,m}x\Pi_{m,M}-a\eins_M)=\det(x-a\eins_M)
\end{equation}
for any $a\in\mathbb{C}$. In the first equality, we have exploited the fact $g$ can be written as product of unitary matrix and the projection $\Pi_{M,m}$ when $M=l$.
\end{example}

\section{Conclusions}\label{sec:conclusio}

We extended the harmonic analysis approach to products of complex rectangular matrices and Hermitian matrices of a fixed rank. As already experienced for products of real asymmetric and real anti-symmetric matrices~\cite{KFI2019}, we need two different spherical functions and, hence, transforms, one for each matrix space. They slightly differ from the original definition by Harish-Chandra et al., see~\cite{Helgason2000} and references therein, in the normalization and the duplication of the complex plane of the ``Mellin-Fourier'' (frequency) parameter $s$. In total, one needs $2^r$ copies for a matrix of rank $r$ denoted by the vector $L\in\mathbb{Z}_2^r\simeq\{0,1\}^r$. These copies are essential to keep the information of the number of positive and negative eigenvalues, since this number stays fixed in such a product as already observed in~\cite{FIL2018}. In spite of these two modifications of the spherical transform, the results resemble very much the original results for matrices on the general (special) linear group~\cite{Helgason2000,KK2016}.

We applied this theoretical framework to a product of a P\'olya ensemble on the complex rectangular matrices and a 
fixed Hermitian matrix as well as a random Hermitian matrix drawn from a polynomial ensemble. All matrices have a specific rank that does not need to be maximal. Both cases yield again polynomial ensembles.  We computed their joint probability densities of the eigenvalues, sets of bi-orthonormal functions and their kernel. For the case of a random Hermitian matrix it is noteworthy to say that the contour integral representations of the bi-orthonormal functions and kernels look extremely similar to those already found for products of complex matrices~\cite{KK2019} and sums on the classical Lie-algebras~\cite{Kieburg2017}. Those formulas are an ideal basis to start a large $n$-analysis. A work on the hard edge statistics of the products of multiplicative P\'olya ensembles with a GUE is currently in preparation~\cite{Kieburg2019}.

What is still puzzling, even disturbing, is the rather different generalization of the spherical function when comparing the case of real antisymmetric matrices~\cite[Equation~(2.11)]{KFI2019} and of Hermitian matrices, see Eq.~\eqref{spher:as}. In the former we only  omitted each second frequency $s_j$, since they correspond to vanishing determinants, while in the latter we even needed to extend the frequency space to an additional parameter set $L$. Thus, harmonic analysis on specific representations of Lie groups seems to avoid a simple unified approach. Regarding this point, we should mention that the corresponding harmonic analysis for the adjoint action of the general linear groups on  the Lie algebras of the orthogonal matrices of odd dimension and of unitary symplectic matrices are still open. Maybe when these gaps are filled, one can easier identify the proper framework that encompasses all these cases and, hopefully, even more.

\section*{Acknowledgements}

I am grateful for the fruitful discussions with Holger K\"osters. Additionally, I would like to thank Peter Forrester and Dang-Zheng Liu for reading the first draft of this work.

\appendix

\section{Proofs of Sec.~\ref{sec:spherical}}

In this appendix, we provide the proofs of Theorem~\ref{thm:spher.func} in Subsection~\ref{Proof:spher.func}, of Proposition~\ref{prop:fact.spher} in Subsection~\ref{Proof:fact.spher}, and of Proposition~\ref{prop:inv} in Subsection~\ref{Proof:inv}. Most ideas follow those employed in~\cite{Kieburg2017} and~\cite{KFI2019}.

\subsection{Proof of Theorem~\ref{thm:spher.func}}\label{Proof:spher.func}

Due to the unitary invariance of the spherical function, we can diagonalize $x=\hat{k}\,\diag(a,0,\ldots,0)\hat{k}^*\in H_l^{(n)}$ with $a\in D_n$  and absorb the diagonalizing  unitary matrix in the Haar distributed matrix $k\in K_n$ in Eq.~\eqref{spher:as}. To proceed further, we first consider the case $\tilde{a}\in H_l^{(l)}=H_l$ and, afterwards, send $\tilde{a}\to\diag(a,0,\ldots,0)$ for $H_l^{(n)}$ with $l>n$  after we have set $\tilde{s}=\diag(s+(l-n)\eins_n,s^{(l-n)})$ and $\tilde{L}=\diag(L+(l-n)\eins_n,L^{(l-n)})$. The particular choice of $\tilde{s}$ and $\tilde{L}$  guarantees that the following product in the definition of $\Phi$,
\begin{equation}\label{proof:spher.func.0}
\prod_{j=n+1}^{l} {\rm sign}[\det \Pi_{j,l-1}\tilde{k}x'(k)\tilde{k}^*\Pi_{l-1,j}]^{\tilde{L}_j-\tilde{L}_{j+1}-1} |\det \Pi_{j,l-1}\tilde{k}\tilde{a}\tilde{k}^*\Pi_{l-1,j}|^{\tilde{s}_j-\tilde{s}_{j+1}-1},
\end{equation}
is set to unity, because it would vanish for matrices of rank $n$.

In the first step, we construct a recursion in $l$ for ${\rm Re}\,(\tilde{s}_j-\tilde{s}_{j+1})\geq 2$ for all $j=1,\ldots,l$. To achieve this goal, we shift $k\to\diag(\tilde{k},1)k$ with an auxiliary Haar distributed $\tilde{k}\in K_{l-1}$. Then we have
\begin{equation}\label{proof:spher.func.1}
\begin{split}
	&\Phi(\tilde{s},\tilde{L}; \tilde{a})\\
	=&[{\rm sign}(\det \tilde{a})]^{\tilde{L}_l} |\det \tilde{a}|^{\tilde{s}_l}\int_{K_l}d^*k\int_{K_{l-1}}d^*\tilde{k}\prod_{j=1}^{l-1} {\rm sign}[\det \Pi_{j,l-1}\tilde{k}x'(k)\tilde{k}^*\Pi_{l-1,j}]^{\tilde{L}_j-\tilde{L}_{j+1}-1} |\det \Pi_{j,l-1}\tilde{k}x'(k)\tilde{k}^*\Pi_{l-1,j}|^{\tilde{s}_j-\tilde{s}_{j+1}-1}\\
	=&[{\rm sign}(\det \tilde{a})]^{\tilde{L}_n} |\det \tilde{a}|^{\tilde{s}_n}\int_{K_l}d^*k\Phi(\diag(\tilde{s}_1,\ldots,\tilde{s}_{l-1})-(\tilde{s}_l+1)\eins_{l-1},\diag(\tilde{L}_1,\ldots,\tilde{L}_{l-1})-(\tilde{L}_l+1)\eins_{l-1};x'(k))
\end{split}
\end{equation}
with the co-rank $1$ projection $x'(k)=\Pi_{l-1,l}k\tilde{a}k^*\Pi_{l,l-1}\in H_{l-1}$. Here, we have exploited $\Pi_{j,l}=\Pi_{j,l-1}\Pi_{l-1,l}$ for all $j=1,\ldots,l-1$ and $\Pi_{l-1,l}\diag(\tilde{k},1)=\tilde{k}\Pi_{l-1,l}$.  Now we can use again the unitary invariance of the spherical transform under $K_{l-1}$. The $l-1$ eigenvalues $a'$ of $x'(k)$ are distributed as follows.

\begin{theorem}[Distribution of a Co-Rank $1$ Projected Hermitian Matrix, see~{\cite[Proposition~4.2]{Baryshnikov2001}}]\label{thm:recursion}\

Let $\tilde{a}\in D_l$ be non-degenerate and fixed and $k\in K_n$ be Haar distributed. Then the eigenvalues $a'\in D_{l-1}$ of the $(l-1)\times(l-1)$ random matrix $x'(k)=\Pi_{l-1,l}k\tilde{a}k^*\Pi_{l,l-1}$ are distributed by
  \begin{equation}\label{proj}
  p(a'|\tilde{a}) =\frac{\Delta_{l-1}(a')}{\Delta_l(\tilde{a})} 
  \det   \left [ \begin{array}{c} \mathbf{1}_l \\ \Theta (a'_j)-\Theta ( \tilde{a}_k - a'_j)
   \end{array}\right]_{\substack{j=1,\dots, l - 1 \\ k =1,\dots, l}},
  \end{equation}
  with $\Theta$ the Heaviside step function and $\mathbf{1}_l=(1,\ldots,1)\in\mathbb{R}^l$.
\end{theorem}

\begin{remark}
We would like to point out that the determinant is equivalent with the interlacing condition~\cite{Baryshnikov2001} $\tilde{a}_1\leq a'_1\leq \tilde{a}_2\leq\ldots \leq \tilde{a}_{n-1}\leq a'_{n-1}\leq \tilde{a}_n$  after ordering the eigenvalues. This can be readily shown by considering the cases when this interlacing is not given. In this situation the determinant vanishes since at least two rows will be linearly dependent.
\end{remark}

We plug Eq.~\eqref{proj} into Eq.~\eqref{proof:spher.func.1} and obtain
\begin{equation}\label{proof:spher.func.2}
\begin{split}
	\Phi(\tilde{s},\tilde{L}; \tilde{a})=&\frac{[{\rm sign}(\det \tilde{a})]^{\tilde{L}_l} |\det \tilde{a}|^{\tilde{s}_l}}{\Delta_l(\tilde{a})}\int_{\mathbb{R}^{l-1}}da'\Delta_{l-1}(a')\det   \left [ \begin{array}{c} \mathbf{1}_l \\ \Theta (a'_j)-\Theta ( \tilde{a}_k - a'_j)
   \end{array}\right]_{\substack{j=1,\dots, n - 1 \\ k =1,\dots, l}}\\
   &\times\Phi(\diag(\tilde{s}_1,\ldots,\tilde{s}_{l-1})-(\tilde{s}_l+1)\eins_{l-1},\diag(\tilde{L}_1,\ldots,\tilde{L}_{l-1})-(\tilde{L}_l+1)\eins_{l-1};a').
\end{split}
\end{equation}
This is the recursion we are looking for. We want to point out that those $a'$ which have a degenerate spectrum belong to a set of measure zero and can be thus neglected.

In the next step we perform a complete induction. The case $l=1$ obviously yields Eq.~\eqref{spher:as.b} ($\Delta_1(\tilde{s})=\Delta_1(\tilde{a})=1$). Thus, let us assume that Eq.~\eqref{spher:as.b} is true for $l-1$. Then we have
\begin{equation}\label{proof:spher.func.3}
\begin{split}
	\Phi(\tilde{s},\tilde{L}; \tilde{a})=&\frac{[{\rm sign}(\det \tilde{a})]^{\tilde{L}_l} |\det \tilde{a}|^{\tilde{s}_l}}{\Delta_l(\tilde{a})}\int_{\mathbb{R}^{l-1}}da'\Delta_{l-1}(a')\det   \left [ \begin{array}{c} \mathbf{1}_l \\ \Theta (a'_j)-\Theta (\tilde{a}_k - a'_j)
   \end{array}\right]_{\substack{j=1,\dots, l - 1 \\ k =1,\dots, l}}\\
   &\times\left(\prod_{j=0}^{l-2}j!\right)\frac{\det[[{\rm sign}(a'_c)]^{\tilde{L}_b-\tilde{L}_l-1}|a'_c|^{\tilde{s}_b-\tilde{s}_l-1}]_{b,c=1,\ldots,l-1}}{\Delta_{l-1}(a')\Delta_{l-1}(\tilde{s}_1,\ldots,\tilde{s}_{l-1})}\\
   =&\left(\prod_{j=0}^{l-1}j!\right)\frac{[{\rm sign}(\det \tilde{a})]^{\tilde{L}_l} |\det \tilde{a}|^{\tilde{s}_l}}{\Delta_l(\tilde{a})\Delta_{l-1}(\tilde{s}_1,\ldots,\tilde{s}_{l-1})}\det   \left [ \begin{array}{c} \mathbf{1}_l \\ \int_0^{\tilde{a}_k} da'
 [{\rm sign}(a')]^{\tilde{L}_j-\tilde{L}_l-1}|a'|^{\tilde{s}_j-\tilde{s}_l-1}  \end{array}\right]_{\substack{j=1,\dots, l - 1 \\ k =1,\dots, l}},
\end{split}
\end{equation}
where we used a generalized version of Andrei\'ef's integral~\cite{KG2010} in the second equality. The integral in the determinant can be calculated as follows
\begin{equation}\label{proof:spher.func.4}
\begin{split}
\int_0^{\tilde{a}_k} da' [{\rm sign}(a')]^{\tilde{L}_j-\tilde{L}_l-1}|a'|^{\tilde{s}_j-\tilde{s}_l-1} =&[{\rm sign}(\tilde{a}_k)]^{\tilde{L}_j-\tilde{L}_l}\int_0^{|\tilde{a}_k|} da' {a'}^{\tilde{s}_j-\tilde{s}_l-1}=\frac{[{\rm sign}(\tilde{a}_k)]^{\tilde{L}_j-\tilde{L}_l}}{\tilde{s}_j-\tilde{s}_l}|\tilde{a}_k|^{\tilde{s}_j-\tilde{s}_l}.
\end{split}
\end{equation}
We plug this integral into~\eqref{proof:spher.func.3}, and pull the factors $1/(\tilde{s}_j-\tilde{s}_l)$ out which combine with $\Delta_{l-1}(\tilde{s}_1,\ldots,\tilde{s}_{l-1})$ to $(-1)^{l-1}\Delta_l(\tilde{s})$. The sign cancels with permuting the first row $\mathbf{1}_l$ completely through to the last one. Taking the remaining factor ${\rm sign}(\det \tilde{a})]^{\tilde{L}_l} |\det \tilde{a}|^{\tilde{s}_l}$ into the determinant, we find Eq.~\eqref{spher:as.b} for $l=n$ because of ${\rm sign}(\det \tilde{a})=\prod_{j=1}^l{\rm sign}(\tilde{a}_j)$.

What remains for the case $l=n$ is to uniquely extend this result to general $\widetilde{s}\in\mathbb{C}^l$. Here, we make use of Carlson's theorem~\cite{Mehta}. First, we can restrict ourselves to the situation when $|\tilde{a}_j|\leq1$ for all $j=1,\ldots,l$ because we can always use $\Phi(\tilde{s},\tilde{L};\tilde{a})=\tilde{a}_{\max}^{\bar{s}}\Phi(\tilde{s},\tilde{L}; \tilde{a}/\tilde{a}_{\max})$ with $\tilde{a}_{\max}=\max_{j=1,\ldots,l}\{|\tilde{a}_j|\}$ and $\bar{s}=\sum_{j=1}^l\tilde{s}_j$ when the spectrum of $\tilde{a}$ exceeds the unit circle. Then the integrand of the definition~\eqref{spher:as} is bounded and holomorphic  on the positive real half-plane for all $l$ variables $\delta \tilde{s}_j=\tilde{s}_j-\tilde{s}_{j+1}-1$. The same is true for the right hand side of Eq.~\eqref{spher:as.b}. Therefore, Carlson's theorem tells us that we can uniquely analytically extend the variables $\delta \tilde{s}_j$ to the whole complex plane for this equation. In particular, the equation~\eqref{spher:as.b} is true for all complex $\tilde{s}_j$ excluding the poles. This closes the proof for the case $l=n$.

Next, we set $\tilde{s}=\diag(s+(l-n)\eins_n,s^{(l-n)})$ and $\tilde{L}=\diag(L+(l-n)\eins_n,L^{(l-n)})$ in the proven identity for $\Phi(\tilde{s},\tilde{L}; \tilde{a})$ and then take the limit  $\tilde{a}\to\diag(a,0,\ldots,0)$. This leads to
\begin{equation}\label{proof:spher.func.5}
\begin{split}
&\Phi(\diag(s+(l-n)\eins_n,s^{(l-n)}),\diag(L+(l-n)\eins_n,L^{(l-n)}); \diag(a,0,\ldots,0))\\
=&\left(\prod_{j=0}^{n-1}\frac{(l-j)!\Gamma[s_j+1]}{\Gamma[s_j+l-n+1]}\right)\frac{\det[[{\rm sign}(a_c)]^{L_b}|a_c|^{s_b}]_{b,c=1,\ldots,n}}{\Delta_n(a)\Delta_n(s)}
\end{split}
\end{equation}
On the other hand, the spherical function is with these values equal to
\begin{equation}\label{proof:spher.func.6}
\begin{split}
&\Phi(\diag(s+(l-n)\eins_n,s^{(l-n)}),\diag(L+(l-n)\eins_n,L^{(l-n)}); \diag(a,0,\ldots,0))\\
=&\int_{K_l}d^*k\prod_{j=1}^{n} {\rm sign}[\det \Pi_{j,l}kxk^*\Pi_{l,j}]^{L_j-L_{j+1}-1} |\det \Pi_{j,l}kxk^*\Pi_{l,j}|^{s_j-s_{j+1}-1}\\
=&C_{l,n}(s)\Phi(s,L; \Pi_{l,n}a\Pi_{n,l}).
\end{split}
\end{equation}
A comparison of both formulas yields Eq.~\eqref{normalization.b} for the normalization $C_{l,n}(s)$ when choosing $a=\eins_n$, because $\Phi(s,L; \Pi_{l,n}\Pi_{n,l})=1$, and the result~\eqref{spher:as.b} when dividing the general result by $C_{l,n}(s)$.

Finally, the spherical function $\Psi(s;g)$ for a $g\in G_{l,m}^{(n)}$ immediately follows from the definition~\eqref{spher:s} and the identity~\eqref{spher:as.b}. This finishes the proof of Theorem~\ref{thm:spher.func}.

\subsection{Proof of Propositions~\ref{prop:fact.spher}}\label{Proof:fact.spher}

The factorization formula~\eqref{factorization.rect} for $\Psi$ immediately follows from the one for $\Phi$. Thence, we only need to prove the latter.

We choose a $g\in G_{l,m}^{(n_1)}$ and an $x\in H_m^{(n_2)}$  with $r=\min\{n_1,n_2\}$ the rank of the product $g_1g_2$ and consider $\delta s_j=s_j-s_{j+1}-1$ with $\RE\,\delta s_j>0$ for $j=1,\ldots, r$. As in the proof of Theorem~\ref{thm:spher.func} we can extend the result to the whole complex plane with the help of Carlson's theorem~\cite{Mehta}.  Thus, we omit this part and concentrate on proving Eq.~\eqref{factorization} for the situation with this restriction.

Considering the integral
\begin{equation}\label{proof:fact.spher.1}
\begin{split}
C_{l,r}(s)\int_{K_m}d^*k\Phi(s,L; gkxk^*g^*)=&\int_{K_m}d^*k\int_{K_l}d^*k' \prod_{j=1}^{r} {\rm sign}[\det \Pi_{j,l}k'gkxk^*g^*{k'}^*\Pi_{l,j}]^{L_j-L_{j+1}-1}\\
&\times|\det \Pi_{j,l}k'gkxk^*g^*{k'}^*\Pi_{l,j}|^{s_j-s_{j+1}-1},
\end{split}
\end{equation}
we perform a QR-decomposition of the rectangular matrix $k'g=t\Pi_{l,m}\tilde{k}$ with $t\in T_l$ and $\tilde{k}\in K_m$. Then we can exploit that $\Pi_{j,l} t=t_j\Pi_{j,l}$ with $t_j=\Pi_{j,l} t\Pi_{l,j}$ is also a lower triangular matrix, namely exactly the $j\times j$ upper left block of $t$. Moreover, the unitary matrix $\tilde{k}$ can be absorbed in the integration over $k\in K_m$ because of the invariance of the Haar measure. Collecting everything we have
\begin{equation}\label{proof:fact.spher.2}
\begin{split}
C_{l,r}(s)\int_{K_m}d^*k\Phi(s,L; gkxk^*g^*)=&\int_{K_m}d^*k\int_{K_l}d^*k' \prod_{j=1}^{r} {\rm sign}[\det t_j\Pi_{j,m}kxk^*\Pi_{m,j}t_j^*]^{L_j-L_{j+1}-1}\\
&\times|\det t_j\Pi_{j,m}kxk^*\Pi_{m,j}t_j^*|^{s_j-s_{j+1}-1}\\
=&\int_{K_l}d^*k'\prod_{j=1}^{r}[\det t_jt_j^*]^{s_j-s_{j+1}-1}\int_{K_m}d^*k\\
&\times \prod_{j=1}^{r} {\rm sign}[\det \Pi_{j,m}kxk^*\Pi_{m,j}]^{L_j-L_{j+1}-1}|\det \Pi_{j,m}kxk^*\Pi_{m,j}|^{s_j-s_{j+1}-1},
\end{split}
\end{equation}
where we employed $\Pi_{j,l}\Pi_{l,m}=\Pi_{j,m}$  for $j\leq r\leq l,m$.
We underline that $t_j$ only depends on $g$ and $k'$ but not on $k$ or $x$. When using $t_jt_j^*=\Pi_{j,l} tt^*\Pi_{l,j}=\Pi_{j,l} k'gg^*{k'}^*\Pi_{l,j}$ for all $j=1,\ldots,m$ and then normalize the integrals properly so that we identify them with the spherical transforms of $g$ and $x$, we find the claim.
The shifts in the $s$ and $L$ for each of the two cases result from the fact that $s_{r+1}=\max\{n_1,n_2\}-r-1$ and $L_{r+1}={\rm mod}_2(\max\{n_1,n_2\}-r-1)$ instead of $-1$ for both, according to the Definition~\eqref{def:spherical-functions}.

For the case of $n_2=r$, we encounter the quotient $C_{l,n_1}(\tilde{s})\ C_{m,n_2}(s)/C_{l,n_2}(s)$ which is equal to $C_{m,n_1}(\tilde{s})$, where $\tilde{s}=\diag(s+(n_1-n_2)\eins_{n_1-n_2},s^{(n_1-n_2)})$.
This ends the proof.

\subsection{Proof of Proposition~\ref{prop:inv}}\label{Proof:inv}

Since $A_n$ is a subset of $D_n$ we can concentrate us on proofing the inverse of the spherical transform of $S_\Phi$. Indeed when comparing the second lines of Eqs.~\eqref{Strafo:as} and~\eqref{Strafo:s} it becomes clear that the two transformations are identical for the domain $A_n$ since the $L$ depended part drops out. Therefore we also do not need to sum over $L\in\mathbb{Z}_2^n$ in the inverse.

We only need to show that Eq.~\eqref{Strafo:as:inv} holds for any $p_D\in L^{1,\mathbb{S}}(D_n)$ and when the eigenvalues $a\in D_n$ are non-degenerate. The inverse of the spherical transform $\mathcal{S}_{\Phi}$ of $\mathcal{S}_{\Phi}p_D$ is explicitly given by
\begin{equation}
\begin{split}
\mathcal{S}_{\Phi}^{-1}[\mathcal{S}_{\Phi}p_D](a)=&\sum_{L\in\{0,1\}^n}\frac{\Delta_n(a)}{(n!)^2}\lim_{\epsilon\to0}\int_{\mathbb{R}^n} \frac{ds}{(4\pi)^{n}}\left(\prod_{l=1}^n\zeta_n(\epsilon s_l)\right)\Delta_n(\imath s+s^{(n)})\\
&\times\det[[{\rm sign}(a_c)]^{L_b}|a_c|^{-\imath s_b-n+b-1}]_{b,c=1,\ldots,n}\biggl(\int_{D_n}d\tilde{a} p_D(\tilde{a})\frac{\det[[{\rm sign}(\tilde{a}_c)]^{L_b}|\tilde{a}_c|^{\imath s_b+n-b}]_{b,c=1,\ldots,n}}{\Delta_n(\tilde{a})\Delta_n(\imath s+s^{(n)})}\biggl).
\end{split}
\end{equation}
Both integrals over $s$ and $\tilde{a}$ are absolutely integrable. The integrability for $s$ is given due to the regularization with $\zeta$ and that the spherical transform $\mathcal{S}_{\Phi}P_H(\imath s+s^{(n)},L)$ is bounded on the integration domain  $s\in\mathbb{R}^n$ by $\int dx |P_H(x)|<\infty$. The integrability of $\tilde{a}$ results from the facts that $p_D$ is an $L^1$-function and the
term $\det[[{\rm sign}(\tilde{a}_c)]^{L_b}|\tilde{a}_c|^{\imath s_b+n-b}]_{b,c=1,\ldots,n}/\Delta_n(\tilde{a})$ is bounded on $\tilde{a}\in D_n$. The latter can be seen by noticing that its modulus is homogeneous in $\tilde{a}$ of order zero and that the poles of the denominator at the points where $\tilde{a}$ degenerates are compensated by the numerator. Thus we can interchange the integrals as long as the limit $\epsilon\to0$ stays in front of both integrals.

Cancelling some terms,  we find
\begin{equation}
\begin{split}
\mathcal{S}_{\Phi}^{-1}[\mathcal{S}_{\Phi}p_D](a)=&\frac{\Delta_n(a)}{(n!)^2}\lim_{\epsilon\to0}\int_{D_n}\frac{d\tilde{a}}{\Delta_n(\tilde{a})} p_D(\tilde{a})\sum_{L\in\{0,1\}^n}\int_{\mathbb{R}^n} \frac{ds}{(4\pi)^{n}}\left(\prod_{l=1}^n\zeta_n(\epsilon s_l)\right)\\
&\times\det[[{\rm sign}(a_c)]^{L_b}|a_c|^{-\imath s_b-n+b-1}]_{b,c=1,\ldots,n}\det[[{\rm sign}(\tilde{a}_c)]^{L_b}|\tilde{a}_c|^{\imath s_b+n-b}]_{b,c=1,\ldots,n}.
\end{split}
\end{equation}
The integral over $s$ and the sum over $L$ can be done with the aid of Andr\'eief's identity~\cite{An86},
\begin{equation}
\begin{split}
\mathcal{S}_{\Phi}^{-1}[\mathcal{S}_{\Phi}p_D](a)=& \frac{\Delta_n(a)}{n!\det|a|}\lim_{\epsilon\to0}\int_{D_n}\frac{d\tilde{a}}{\Delta_n(\tilde{a})} p_D(\tilde{a})\det\left[\int_{-\infty}^\infty \frac{ds}{2\pi}\zeta_n(\epsilon s)\frac{1+{\rm sign}(a_b\tilde{a}_c)}{2}\left|\frac{\tilde{a}_c}{a_b}\right|^{\imath s+n-b}\right]_{b,c=1,\ldots,n}\\
=& \frac{\Delta_n(a)}{n!\det|a|}\lim_{\epsilon\to0}\int_{D_n}\frac{d\tilde{a}}{\Delta_n(\tilde{a})} p_D(\tilde{a})\det\left[\frac{1+{\rm sign}(a_b\tilde{a}_c)}{2}\left|\frac{\tilde{a}_c}{a_b}\right|^{n-b}\frac{1}{\epsilon}\mathcal{F}\zeta_n\left(\frac{1}{\epsilon}{\rm ln}\left|\frac{\tilde{a}_c}{a_b}\right|\right)\right]_{b,c=1,\ldots,n},
\end{split}
\end{equation}
where the inverse Fourier transform of the regularizing function~\eqref{xi.def} is
\begin{equation}
\mathcal{F}\zeta_n(u)=\frac{1}{2\pi}\int_{-\infty}^\infty dz \zeta_n(z) e^{-\imath z u}=c\Theta(1-u^2)\cos^{2n-1}\left(\frac{\pi u}{2}\right).
\end{equation}
We recall that $\Theta(x)$ denotes the Heaviside step function.
The exact value of the constant $c\neq 0$ is not important. It only correctly normalizes $\mathcal{F}\zeta_n$ because $\zeta_n(0)=1$.

Due to the regularization the integration domain of $\tilde{a}\in D_n$ shrinks to $(\bigcup_{j=1}^n[-|a_j|e^{\epsilon},-|a_j|e^{-\epsilon}]\cup[|a_j|e^{-\epsilon},|a_j|e^{\epsilon}])^n$. The factor incorporating the signs projects to those intervals where $\tilde{a}$ has the same signs as $a$. Thus we only integrate over $\widehat{D}_\epsilon=\bigcup_{\omega\in\mathbb{S}}(\bigcup_{j=1}^n[a_{\omega(j)}e^{-{\rm sign}(a_{\omega(j)})\epsilon},a_{\omega(j)}e^{{\rm sign}(a_{\omega(j)})\epsilon}])^n$. We recall that the eigenvalues $a\in D_n$ are chosen to be non-degenerate. Therefore there is an $\epsilon_0>0$ such that $1/\Delta_n(\tilde{a})$ has no poles for all $0<\epsilon\leq\epsilon_0$, in particular it is uniformly bounded on $\widehat{D}_{\epsilon_0}$ and, thus, on $\widehat{D}_\epsilon\subset \widehat{D}_{\epsilon_0}$ for any $\epsilon\leq\epsilon_0$. Hence, we may expand the determinant in the numerator and get a factor $n!$ due to the symmetry of the integrand, 
\begin{equation}
\begin{split}
\mathcal{S}_{\Phi}^{-1}[\mathcal{S}_{\Phi}p_D](a)=& \lim_{\epsilon\to0}\frac{\Delta_n(a)}{\det |a|}\prod_{l=1}^n\int_{a_le^{-{\rm sign}(a_l)\epsilon}}^{a_le^{{\rm sign}(a_l)\epsilon}} d\tilde{a}_l\ \frac{p_D(\tilde{a})}{\Delta_n(\tilde{a})}\left(\prod_{j=1}^n\left|\frac{\tilde{a}_j}{a_j}\right|^{j-1}\frac{c}{\epsilon}\cos^{2n-1}\left(\frac{\pi }{2\epsilon}{\rm ln}\left|\frac{\tilde{a}_j}{a_j}\right|\right)\right)\\
=& \lim_{\epsilon\to0}\Delta_n(a)\prod_{l=1}^n\int_{|a_l|e^{-{\rm sign}(a_l)\epsilon}}^{|a_l|e^{{\rm sign}(a_l)\epsilon}} \frac{d\tilde{a}_l}{|a_l|} \frac{p_D({\rm sign}(a)|\tilde{a}|)}{\Delta_n({\rm sign}(a)|\tilde{a}|)}\left(\prod_{j=1}^n\left|\frac{\tilde{a}_j}{a_j}\right|^{j-1}\frac{c}{\epsilon}\cos^{2n-1}\left(\frac{\pi }{2\epsilon}{\rm ln}\left|\frac{\tilde{a}_j}{a_j}\right|\right)\right).
\end{split}
\end{equation}
The substitution $u_l={\rm sign}(a_l)/\epsilon\ {\rm ln}|\tilde{a}_l/a_l|$ yields
\begin{equation}
\begin{split}
\mathcal{S}_{\Phi}^{-1}[\mathcal{S}_{\Phi}p_D](a)=&\Delta_n(a)\lim_{\epsilon\to0}\prod_{l=1}^n\int_{-1}^{1} d\tilde{a}_l\ \frac{p_D(a\exp[{\rm sign}(a)\epsilon u])}{\Delta_n(a\exp[{\rm sign}(a)\epsilon u])}\left(\prod_{j=1}^nc\cos^{2n-1}\left(\frac{\pi u_j }{2}\right)e^{{\rm sign}(a_j)\epsilon j u_j}\right).
\end{split}
\end{equation}
The function $p_D$ is an $L^1$-function on $D_n$. The integrand is therefore an $L^1$-function on $[-1,1]^n$. With the same arguments as in the proof of~\cite[Lemma~2.6]{KK2016} we have
\begin{equation}
\begin{split}
\mathcal{S}_{\Phi}^{-1}[\mathcal{S}_{\Phi}p_D](a)=&\Delta_n(a)\lim_{\epsilon\to0}\prod_{l=1}^n\int_{-1}^{1} d\tilde{a}_l\ \frac{p_D(a)}{\Delta_n(a)}\left(\prod_{j=1}^nc\cos^{2n-1}\left(\frac{\pi u_j }{2}\right)e^{{\rm sign}(a_j)\epsilon j u_j}\right)=p_D(a)
\end{split}
\end{equation}
for those $a\in D_n$ that satisfy
\begin{equation}
\lim_{\epsilon\to0}\prod_{l=1}^n\int_{[-1,1]^n} d\tilde{a} |p_D(a+\epsilon\tilde{a})-p_D(a)|=0,
\end{equation}
which are almost all. This completes the proof.

\section{Proofs of Sec.~\ref{sec:conv}}

Herein, we present the two proofs of Theorem~\ref{thm:jpdf.fixed} and Proposition~\ref{prop:stat.poly} in Subsections~\ref{Proof:jpdf.fixed} and~\ref{Proof:stat.poly}, respectively. The basic ideas are very close to those exploited in~\cite{Kieburg2017,KFI2019}.

\subsection{Proof of Theorem~\ref{thm:jpdf.fixed}}\label{Proof:jpdf.fixed}

Let us denote the distribution of the P\'olya ensemble random matrix $g\in G_{l,m}^{(n_1)}$  by $Q_G\in L^{1,K}_{\rm Prob}(G_{l,m}^{(n_1)})$. Then, the distribution $P_H\in L^{1,K}_{\rm Prob}(H_{l}^{(n_1)})$ of the product $\tilde{x}=gxg^*\in H_l^{(n_1)}$ with   $x\in H_m^{(n_2)}$ fixed is certainly
\begin{equation}
P_H(\tilde{x}|x)=\int_{G_{l,m}^{(n_1)}}dg\delta(\tilde{x}-gxg^*)Q_G(g)
\end{equation}
with the Dirac delta function on $H_l^{(n_1)}$ with respect to the measure specified in Sec.~\ref{sec:pre}. The spherical transform thereof is equal to
\begin{equation}
\mathcal{S}_\Phi P_H(s,L|x)=\int_{G_{l,m}^{(n_1)}}dg\Phi(s,L;gxg^*)Q_G(g)=\int_{G_{l,m}^{(n_1)}}dg\int_{K_m}d^*k\Phi(s,L;gkxk^*g^*)Q_G(g),
\end{equation}
where we could introduce the Haar integral over $K_m$ due to the $K$-invariance of $Q_G$. The factorization in Proposition~\ref{prop:fact.spher} tells us that we have to consider two cases from now on.

First we consider the case $n_1\geq n_2$ which yields the spherical transform
\begin{equation}\label{proof.thm2.1}
\begin{split}
\mathcal{S}_\Phi P_H(s,L|x)=&C_{m,n_1}(\tilde{s})\mathcal{S}_\Psi Q_G(\tilde{s})\Phi(s,L;x)\\
=&\left(\prod_{j=1}^{n_2} \frac{\mathcal{M}\omega_{m-n_1}^{(n_1-n_2)}(s_j+1)}{\mathcal{M}\omega_{m-n_1}^{(n_1-n_2)}(n_2-j+1)}\right)\left(\prod_{j=1}^{n_2}\frac{\mathcal{M}\omega(s_j+n_1-n_2+1)}{\mathcal{M}\omega(n_1-j+1)}\right)\frac{\det[[{\rm sign}(a_c)]^{L_b}|a_c|^{s_b}]_{b,c=1,\ldots,n_2}}{\left(\prod_{j=0}^{n_2-1}1/j!\right)\Delta_{n_2}(a)\Delta_{n_2}(s)}
\end{split}
\end{equation}
with $\tilde{s}=\diag(s+(n_1-n_2)\eins_{n_2},s^{(n_1-n_2)})$. In the second line we exploited the result of Example~\ref{ex:proj}, the spherical transform~\eqref{spher.Polya.G} of a P\'olya ensemble on $G_{l,m}^{(n_1)}$, and the explicit form of the the spherical function $\Phi$, see Eq.~\eqref{spher:as.b}. The expression~\eqref{proof.thm2.1} can be simplified by pushing the numerators of the products into the determinant,
\begin{equation}\label{proof.thm2.2}
\begin{split}
\mathcal{S}_\Phi P_H(s,L|x)=&\left(\prod_{j=1}^{n_2} \frac{(j-1)!}{\mathcal{M}\omega_{m-n_1}^{(n_1-n_2)}(n_2-j+1)\mathcal{M}\omega(n_1-j+1)}\right)\\
&\times\frac{\det[[{\rm sign}(a_c)]^{L_b}|a_c|^{s_b}\mathcal{M}\omega_{m-n_1}^{(n_1-n_2)}(s_b+1)\mathcal{M}\omega(s_b+n_1-n_2+1)]_{b,c=1,\ldots,n_2}}{\Delta_{n_2}(a)\Delta_{n_2}(s)}.
\end{split}
\end{equation}
The term inside the determinant is equal to the Mellin transform $\mathcal{M}\tilde{\omega}_\geq(s_b+1,L|a_c)$ of the function
\begin{equation}
\tilde{\omega}_\geq(\tilde{a}|a_c)=\int_{0}^\infty da'_1 \int_0^\infty da'_2 \omega_{m-n_1}^{(n_1-n_2)}(a'_1){a'_2}^{n_1-n_2}\omega(a'_2)\delta(\tilde{a}-a'_1a'_2a_c).
\end{equation}
This time we employ the standard Dirac delta function on the real line whose evaluation leads to Eq.~\eqref{weight.fixed.g}. Here we want to point out the identity $\omega_{m-n_1}^{(n_1-n_2)}(a'_1)={a'_1}^{n_1-n_2}\omega_{m-n_1}^{(0)}(a'_1)$.

Comparison of Eq.~\eqref{proof.thm2.2} with the spherical transform~\eqref{spher.pol.H} of a polynomial ensemble on $H_{l}^{(n_2)}$ yields the first part of Theorem~\ref{thm:jpdf.fixed} due to the invertibility of this transform.

Let us go over to the opposite case $n_1\leq n_2$. For  this case the spherical transform of $P_H$ reads
\begin{equation}\label{proof.thm2.3}
\begin{split}
\mathcal{S}_\Phi P_H(s,L|x)=&C_{m,n_2}(\tilde{s})\mathcal{S}_\Psi Q_G(s)\Phi(\tilde{s},\tilde{L};x)\\
=&\left(\prod_{j=1}^{n_1} \frac{\mathcal{M}\omega_{m-n_2}^{(n_2-n_1)}(s_j+1)}{\mathcal{M}\omega_{m-n_2}^{(n_2-n_1)}(n_1-j+1)}\right)\left(\prod_{j=1}^{n_1}\frac{\mathcal{M}\omega(s_j+1)}{\mathcal{M}\omega(n_1-j+1)}\right)\\
&\times\frac{\prod_{j=0}^{n_2-1}j!}{\Delta_{n_2}(a)\Delta_{n_2}(\tilde{s})}\det\left[\begin{array}{c} [{\rm sign}(a_c)]^{L_b}|a_c|^{s_b} a_c^{n_2-n_1} \\  a_c^{n_2-n_1-d} \end{array}\right]_{\substack{b=1,\ldots,n_1 \\ d=1,\ldots,n_2-n_1 \\ c=1,\ldots,n_2}},
\end{split}
\end{equation}
where we now choose  $\tilde{s}=\diag(s+(n_2-n_1)\eins_{n_1},s^{(n_2-n_1)})$ and $\tilde{L}=\diag(L+(n_2-n_1)\eins_{n_1},L^{(n_2-n_1)})$. The Vandermonde determinant of $\tilde{s}$ is equal to
\begin{equation}
\Delta_{n_2}(\tilde{s})= (-1)^{n_1(n_2-n_1)+(n_2-n_1)(n_2-n_1-1)/2} \left(\prod_{j=0}^{n_2-n_1-1}j!\right)\left(\prod_{j=1}^{n_1}\frac{\Gamma[s_j+n_2-n_1+1]}{\Gamma[s_j+1]}\right)\Delta_{n_1}(s).
\end{equation}
The sign can be cancelled by reshuffling the rows in the determinant at the end of Eq.~\eqref{proof.thm2.3} and the Gamma functions can be combined with the Mellin transform $\mathcal{M}\omega_{m-n_2}^{(n_2-n_1)}$ as follows
\begin{equation}\label{proof.thm2.5}
\begin{split}
\frac{\Gamma[s_j+1]}{\Gamma[s_j+n_2-n_1+1]}\mathcal{M}\omega_{m-n_2}^{(n_2-n_1)}(s_j+1)=&\frac{\Gamma[s_j+1]}{\Gamma[s_j+n_2-n_1+1]}\frac{(m-n_2)!\Gamma[s_j+n_2-n_1+1]}{\Gamma[s_j+m-n_1+1]}\\
=&\frac{(m-n_2)!}{(m-n_1)!}\mathcal{M}\omega_{m-n_1}^{(0)}(s_j+1).
\end{split}
\end{equation}
Again, we pull the Mellin transforms in the denominator into the determinant and obtain
\begin{equation}\label{proof.thm2.6}
\begin{split}
\mathcal{S}_\Phi P_H(s,L|x)=&\left(\prod_{j=1}^{n_1} \frac{(m-n_2)!(n_2-j)!}{(m-n_1)!(n_1-j)!\mathcal{M}\omega_{m-n_2}^{(n_2-n_1)}(n_1-j+1)\mathcal{M}\omega(n_1-j+1)}\right)\\
&\times\frac{\prod_{j=0}^{n_1-1}j!}{\Delta_{n_2}(a)\Delta_{n_1}(s)}\det\left[\begin{array}{c} a_c^{b-1} \\ 
{[{\rm sign}(a_c)]^{L_d}}|a_c|^{s_d}  a_c^{n_2-n_1}\mathcal{M}\omega_{m-n_1}^{(0)}(s_d+1)\mathcal{M}\omega(s_d+1)  \end{array}\right]_{\substack{b=1,\ldots,n_2-n_1 \\ d=1,\ldots,n_1 \\ c=1,\ldots,n_2}}.
\end{split}
\end{equation}
The prefactor can be simplified as in~\eqref{proof.thm2.5}. The function in the last rows of the determinant are again a Mellin transform; this time it is of the function
\begin{equation}
\tilde{\omega}_<(\tilde{a}|a_c)= a_c^{n_2-n_1} \int_{0}^\infty da'_1 \int_0^\infty da'_2 \omega_{m-n_1}^{(0)}(a'_1)\omega(a'_2)\delta(\tilde{a}-a'_1a'_2a_c).
\end{equation}
It agrees with Eq.~\eqref{weight.fixed.l} after evaluating the Dirac delta function.

The inverse of the spherical transform, see Eq.~\eqref{Strafo:as:inv}, can be carried out via a generalized form of Andr\'eief's identity~\cite{KG2010} leading immediately to the result~\eqref{pfixed.l}. Indeed, this density is also a polynomial ensemble for which we make use of the elementary symmetric polynomials~\eqref{elementary} and the notation $a_{\neq j}=\diag(a_1,\ldots,a_{j-1},a_{j+1},\ldots,a_{n_2})\in D_{n_2-1}$. Those polynomials satisfy the bi-orthogonality relation
\begin{equation}\label{elem.orth}
\begin{split}
\sum_{c=1}^{n_2} a_c^{b-1}\frac{e_{n_2-b'}(-a_{\neq c})}{\prod_{h\neq c}(a_c-a_h)}=&\oint\frac{dz}{2\pi i z^{b'}} \sum_{c=1}^{n_2} a_c^{b-1} \prod_{h\neq c}\frac{a_h-z}{a_h-a_c}\\
=&\frac{1}{\Delta_{n_2}(a)}\oint\frac{dz}{2\pi i (-z)^{b'}} \sum_{c=1}^{n_2} (-1)^{c-1} a_c^b \det\left[\begin{array}{c} (-z)^{d-1} \\ a_f^{d-1} \end{array}\right]_{\substack{ f\neq c \\ d=1,\ldots,n_2 }}\\
=&\frac{(-1)^{b'}}{\Delta_{n_2}(a)}\sum_{c=1}^{n_2} (-1)^{c-1} a_c^{b-1} \det\left[\begin{array}{c} a_f^{d-1} \end{array}\right]_{\substack{ f\neq c \\ d\neq b'}}\\
=&\frac{(-1)^{b'}}{\Delta_{n_2}(a)}\det\left[a_f^{b-1}\ |\ a_f^{d-1} \right]_{\substack{  f=1,\ldots, n_2 \\ d\neq b'}}\\
=&\delta_{bb'}
\end{split}
\end{equation}
for all $b,b'=1,\ldots,n_2$. With the help of this relation, we extend the joint probability density by the determinant
\begin{equation}
\begin{split}
\det\left[\frac{e_{n_2-b'}(-a_{\neq c})}{\prod_{h\neq c}(a_c-a_h)}\right]_{c,b'=1,\ldots, n_2}=&\det\left[(-1)^{b'+c}\frac{\det\left[\begin{array}{c} a_f^{d-1} \end{array}\right]_{\substack{ f\neq c \\ d\neq b'}}}{\Delta_{n_2}(a)}\right]_{c,b'=1,\ldots, n_2}=\frac{1}{\det[a_c^{b-1}]_{c,b'=1,\ldots, n_2}}=\frac{1}{\Delta_{n_2}(a)}
\end{split}
\end{equation}
and multiply it with the other determinant where we employ the well-known rule $\det B\det C=\det BC$ of two equal sized matrices $B$ and $C$. This yields in the first $n_2-n_1$ rows the Kronecker delta $\delta_{bb'}$ in which one can expand the determinant.
This completes the proof.

\subsection{Proof of Proposition~\ref{prop:stat.poly}}\label{Proof:stat.poly}

\paragraph{The case $n_1\geq n_2$}

The proof of the case $n_1\geq n_2$ is based on the identity
\begin{equation}\label{proof.stat.poly.1}
\begin{split}
\int_{-\infty}^\infty d\tilde{a} \tilde{a}^{b}\tilde{\omega}_\geq(\tilde{a}|a)=&\int_{0}^\infty da'_1 \int_0^\infty da'_2 \omega_{m-n_1}^{(n_1-n_2)}(a'_1){a'_2}^{n_1-n_2}\omega(a'_2)(a'_1a'_2a)^{b}\\
=&\frac{(m-n_1)!(n_1+b-n_2)!}{(m+b-n_2)!}\mathcal{M}\omega(n_1+b-n_2+1) a^{b}.
\end{split}
\end{equation}
Thence, the polynomials
\begin{equation}
\tilde{p}_j(\tilde{a})=\oint \frac{dz}{2\pi i z}\chi_{M}^\geq(z)p_j\left(\frac{\tilde{a}}{z}\right)=\sum_{h=0}^{j}\frac{(m+h-n_1)!}{(m-n_1)!j!\mathcal{M}\omega(h+1)}c_{jh}\tilde{a}^h,
\end{equation}
with $c_{jh}$ the coefficients of the polynomial $p_j(a)$, compensate exactly those constants in front of the right hand side of Eq.~\eqref{proof.stat.poly.1}, meaning
\begin{equation}\label{proof.stat.poly.2}
\begin{split}
\int_{-\infty}^\infty d\tilde{a} \tilde{p}_j(\tilde{a})\tilde{\omega}_\geq(\tilde{a}|a)=&p_j(a).
\end{split}
\end{equation}
This can be exploited to show the bi-orthonormality
\begin{equation}
\begin{split}
\int_{-\infty}^\infty d\tilde{a} \tilde{p}_b(\tilde{a})\tilde{q}_{b'}(\tilde{a})=&\int_{-\infty}^\infty d\tilde{a} \tilde{p}_b(\tilde{a})\int_{-\infty}^\infty da \tilde{\omega}_\geq(\tilde{a}|a)q_{b'}(a=\int_{-\infty}^\infty da p_b(a)q_{b'}(a)=\delta_{bb'}
\end{split}
\end{equation}
for all $b,b'=0,\ldots,n_2-1$. The integrals can be interchanged because the integrand is absolutely integrable.

For the kernel we plug the definitions of $\tilde{p}_j$ and $\tilde{q}_j$ into Eq.~\eqref{kernel} for $\tilde{K}_{n_2}$. Additionally, we performed the change of coordinates $a\to\tilde{a}_2/a$. Afterwards, we interchange the two integrals  with the sum because it is finite and find the claim.

\paragraph{The case $n_1< n_2$}

This time we start from the identity
\begin{equation}\label{proof.stat.poly.3}
\begin{split}
\int_{-\infty}^\infty d\tilde{a} \tilde{a}^{b}\tilde{\omega}_<(\tilde{a}|a)=&a^{n_2-n_1}\int_{0}^\infty da'_1 \int_0^\infty da'_2\omega_{m-n_1}^{(0)}(a'_1)\omega(a'_2)(a'_1a'_2a)^{b}\\
=&\frac{(m-n_1)!b!}{(m+b-n_1)!}\mathcal{M}\omega(b+1) a^{n_2-n_1+b}
\end{split}
\end{equation}
for $n_1<n_2$. Moreover, we exploit the orthogonality of the weights  $\{q_{n_2-n_1+j}\}_{j=0,\ldots,n_1-1}$ to all polynomials of order $n_2-n_1-1$. Assuming $p_j(a)=\sum_{h=0}^j c_{jh}a^h$, then the following bi-orthonormality holds true
\begin{equation}
\int_{-\infty}^\infty da\, p_{n_2-n_1+b}(a)q_{n_2-n_1+b'}(a)=\int_{-\infty}^\infty da\, \sum_{h=0}^{b}c_{n_2-n_1+b,n_2-n_1+h}a^{n_2-n_1+h}\ q_{n_2-n_1+b'}(a)=\delta_{bb'}
\end{equation}
for all $b,b'=0,\ldots,n_1-1$. From now on the line of reasoning is exactly the same as for the case $n_1\geq n_2$ since the polynomials $\tilde{p}_j$ comprise those coefficients that cancel with the factors in~\eqref{proof.stat.poly.3} leading to the the truncation  $\sum_{h=0}^{b}c_{n_2-n_1+b,n_2-n_1+h}a^{n_2-n_1+h}$ of the original polynomials $p_j$ when integrating over $\tilde{\omega}_<(\tilde{a}|a)$. Thence, the bi-orthonormality is a consequence of the one between $p_j$ and $q_j$.

The formula~\eqref{ppoly.l} for the kernel is also directly found after interchanging the integrals with the finite sum. This finishes the proof of the proposition.

\end{document}